\documentclass[11pt]{article}
\usepackage{etex}
\usepackage{amsmath,amssymb,amsfonts,amsthm}
\usepackage{a4wide}
\usepackage{color}
\usepackage{verbatim}
\usepackage{epsfig,subfigure,epstopdf}
\usepackage[]{graphics}
\usepackage{graphicx,inputenc}
\usepackage{subfigure}
\usepackage[justification=centering]{caption}
\usepackage{pictex}
\usepackage{wrapfig}
\usepackage{booktabs}
\usepackage{multirow}
\usepackage[section]{placeins}

\usepackage[linesnumbered,ruled,vlined, ruled,vlined]{algorithm2e}

\usepackage[colorlinks,linktocpage,linkcolor=blue]{hyperref}
\usepackage{fancyhdr}

\usepackage{algorithmic}

\usepackage[T1]{fontenc}
\usepackage{authblk}
\usepackage[textwidth=6.0in,textheight=9in]{geometry}

\usepackage[colorlinks,linktocpage,linkcolor=blue]{hyperref}
\usepackage{fancyhdr}

\newtheoremstyle{exampstyle}
  {\topsep} 
  {\topsep} 
  {\itshape} 
  {} 
  {\bfseries} 
  {.} 
  {.5em} 
  {} 

\theoremstyle{exampstyle}

\numberwithin{equation}{section}
\newtheorem{lemma}{Lemma}[section]

\newtheorem{theorem}{Theorem}[section]

\newtheorem{remark}{Remark}[section]
\newtheorem{assumption}{Assumption}[section]
\allowdisplaybreaks

\usepackage{parskip}
\let\oldref\ref
\renewcommand{\ref}[1]{(\oldref{#1})}  
\renewcommand{\ref}[1]{(\oldref{#1})}

\newbox\boxaddrone \newbox\boxaddrtwo

\def\N+{n\in\mathbb{N}^{+}}

\def\E{\mathbb{E}}

\def\R{\mathbb{R}}

\def\n{\partial{\overrightarrow{\bf n}}}
\def\B{\mathcal B}

\newcommand{\q}{\quad}    \def\R{{\mathbb R}}

\newtheorem{example}{Example}[section]

\newcommand{\lla}{\|{\hskip -1pt}|}
\newcommand{\rra}{\|{\hskip -1pt}|}

\newcommand{\Om}{\Omega}
\newcommand{\de}{\delta}

\newcommand{\lam}{\lambda}

\newcommand{\vep}{\varepsilon}
\newcommand{\lj}{|{\hskip -1pt} \|}
\newcommand{\rj}{|{\hskip -1pt} \|}

\newcommand{\diam}{\mathrm{diam}}

\newcommand{\sd}{\mathsf{d}}

\newcommand{\be}{\begin{eqnarray}}
\newcommand{\ee}{\end{eqnarray}}
\newcommand{\beq}{\begin{equation}}
\newcommand{\eeq}{\end{equation}}
\newcommand{\ben}{\begin{eqnarray*}}
\newcommand{\een}{\end{eqnarray*}}
\newcommand{\nn}{\nonumber}

\begin{document}

\title{Quantitative estimates for a nonlinear inverse source problem in a coupled diffusion equations with uncertain measurements}

\author[1,2]{Chunlong Sun\thanks{sunchunlong@nuaa.edu.cn}}
\author[3]{Wenlong Zhang\thanks{ zhangwl@sustech.edu.cn}}
\author[4,5]{Zhidong Zhang\thanks{zhangzhidong@mail.sysu.edu.cn}}

\affil[1]{School of Mathematics, Nanjing University of Aeronautics and Astronautics, Nanjing 211106, Jiangsu, China}
\affil[2]{Nanjing Center for Applied Mathematics, Nanjing 211135, Jiangsu, China}
\affil[3]{Department of Mathematics, Southern University of Science and Technology (SUSTech), Shenzhen, Guangdong, China}
\affil[4]{School of Mathematics (Zhuhai), Sun Yat-sen University, Zhuhai 519082, Guangdong, China}
\affil[5]{Guangdong Province Key Laboratory of Computational Science, Sun Yat-sen University, Guangzhou 510000, Guangdong, China}

\maketitle

\begin{abstract}
\noindent This work considers a nonlinear inverse source problem in a coupled diffusion equation from the terminal observation. Theoretically, under some conditions on problem data, we build the uniqueness theorem for this inverse problem and show two Lipschitz-type stability results in $L^2$ and $(H^1(\cdot))^*$ norms, respectively. However, in practice, we could only observe the measurements at discrete sensors, which contain the noise. Hence, this work further investigates the recovery of the unknown source from the discrete noisy measurements. We propose a stable inversion scheme and provide probabilistic convergence estimates between the reconstructions and exact solution in two cases: convergence respect to expectation and convergence with an exponential tail. We provide several numerical experiments to illustrate and complement our theoretical analysis.\\ 

\noindent Keywords: inverse problem, uniqueness, conditional stability, quantitative estimates, numerical inversions.\\ 

\noindent AMS subject classifications: 35R30, 65J20, 65M60, 65N21, 65N30. 
\end{abstract}

\section{Introduction.}
This work focuses on a nonlinear inverse problem in a coupled diffusion systems. Denoting $u_e, u_m$ as the excitation field and emission field, respectively, we consider the coupled system as follows:     
\begin{equation}\label{PDE_ue}
\begin{cases}
\begin{aligned}
(\partial_t-\Delta) u_e(x,t)+(p(x)+q(x))u_e(x,t) &=0, && x \in \Omega\times(0,T],\\
\B u_e(x,t) &= b(x,t), &&x\in\partial\Omega\times(0,T],\\
u_e(x,0)&=0, &&x\in\Omega,\\
\end{aligned}
\end{cases}
\end{equation}
and
\begin{equation}\label{PDE_um}
\begin{cases}
\begin{aligned}
(\partial_t-\Delta) u_m(x,t)+p(x)u_m(x,t) &=q(x)u_e(x,t), && (x,t) \in \Omega\times(0,T],\\
\B u_m(x,t) &=0, &&(x,t)\in\partial\Omega\times(0,T],\\
u_m(x,0)&=0,&&x\in\Omega.
\end{aligned}
\end{cases}
\end{equation}
Here $\Omega\subset\mathbb R^d$ ($1\leq d\leq 3$) is the background medium with smooth boundary $\partial\Omega$, and the boundary condition $\B u=\beta\frac{\partial u}{\n}+u$ with $\beta>0$. Above coupled diffusion system could describe the two diffusion processes in optical tomography, namely, excitation and emission \cite{Arridge99,Arr09,Liu:2020}, where $u_e$ and $u_m$ denote the photon densities of excitation light and emission light, respectively; $p(x)>0$ is the background absorption and $q(x)\ge 0$ is the absorption of the fluorophores in $\Omega$.

In this work, we aim to use the final time data $u_m(x,T)$ to recover the unknown source term $q(x)$. For $q(x)$, we define $\mathcal{Q}:=\{\psi\in C(\bar\Omega): 0\leq \psi \leq M<\infty\}$ as the admissible set, and denote the forward operator $\mathbb{G}: \mathcal{Q}\to H^2(\Omega)$ by 
\begin{equation*}
\mathbb{G}\psi:=u_m(x,T;\psi),\;x\in\Omega.
\end{equation*}
Then the interested inverse problem could be stated as follows: 
\begin{equation}\label{IP-continuous}
\text{with data}\ g(x):=u_m(x,T),\ \text{find}\ q\in {\mathcal Q}\ \text{such that}\ \mathbb{G} q=g.
\end{equation}

In this work, the uniqueness and stability of the nonlinear inverse problem \ref{IP-continuous} would be investigated. Denoting the exact solution by $q^*$, we construct a monotone operator as 
\begin{equation*}\label{operator-K}
K q=\frac{\partial_t u_m(x,T;q)-\Delta g+p(x)g(x)}{u_e(x,T;q)},
\end{equation*}
whose fixed points could generate the desired data $g(x)$. Here the notations $u_e(x,T;q)$ and $u_m(x,T;q)$ reflect the dependence of $u_e$ and $u_m$ on $q$. In Theorem \ref{theorem_uniqueness}, we prove that there is at most one fixed point of $K$, which immediately leads to the uniqueness result of the inverse problem. Next, under some conditions on problem data, we show two Lipschitz-type stability in $L^2$ and $(H^1(\cdot))^*$ norms, respectively (see Theorem \ref{thm-stability}). 

However, in the practical applications of the inverse problem, we could only obtain the noisy datum $g^\sigma$ of $g$ at discrete sensors. Hence, for the numerical reconstruction, we restate the concerned inverse problem in this work as follows:
\begin{equation}\label{IP-numeric}
\text{recovering the unknown $q(x)$ from discrete measurements $\{g^\sigma(x_i)\}_{i=1}^n$}.
\end{equation}
This work will present a quantitative understanding of the convergence in probability of the regularized solutions to inverse problem \ref{IP-numeric}, under the measurement data with random variables.

The inverse source problems may arise from very different applications and modeling, e.g., 
diffusion or groundwater flow processes \cite{new1,AB01, new3, GER83, new2, new6,Isakov2013} 
(\cite{new2} 
recovers the source term from interior measurements, whose motivation lies in the seawater intrusion phenomenon), 
heat conduction or convection-diffusion processes \cite{new4,GER83, liu16, tadi97}, 
or acoustic problems \cite{Badia2011, nelson20}. 
Pollutant source inversion can find many applications, e.g., 
indoor and outdoor air pollution, detecting and monitoring underground water pollution.
Physical, chemical and biological measures have been developed  for the identification of
sources and source strengths \cite{AB01,ZCP01}. Hence, the research on inverse source problems has drawn more and more attention from people and we list several references here. The article \cite{BaoLiZhao:2020} considers the inverse source problems in elastic and electromagnetic waves, and deduce the uniqueness and stability results; the authors in \cite{ChengYamamoto:2022} prove a continuation result and apply it to the inverse source problems in elliptic and parabolic equations; \cite{LiLiWang:2022, WangXuZhao:2024, LassasLiZhang:2023} concern with the inverse random source problems; in \cite{LinZhangZhang:2022, SunZhang:2022, RundellZhang:2020, IsakovLu:2020}, the authors discuss the inverse source problems with boundary measurements. For more works on the inverse source problems, we refer to \cite{LinOuZhangZhang:2024, ZhangWuGuo:2024, FuZhang:2021, Isakov:1990,DingGongLiuLo:2024, JiangLiYamamoto:2024} and the references therein. However, to our knowledge, the nonlinear inverse source problems in coupled diffusion equations such as \ref{IP-continuous} and \ref{IP-numeric} have not been well studied both deterministically and statistically.

In practice, the data is measured at scattered points using sensors, which is contaminated by random measurement errors. Thus it is important to study the discrete random observation model. \cite{Chen-Zhang2022, Chen-Zhang} study a regularized formulation of an inverse source problem and the thin plate spline model with stochastic pointwise measurements and its discretization using the Galerkin FEM. 
This work addresses the inverse problem \ref{IP-numeric} where the available data consists of discrete pointwise measurements contaminated by additive random noise. The measurement model is formulated as:
$$g_i^\sigma=g(x_i)+e_i, \quad i=1, 2, \cdots, n,$$
where the deterministic observational points $\{x_i\}_{i=1}^n$ are distributed quasi-uniformly over the domain $\Omega$; $g(x_i)$
 represents the true (noise-free) value of the unknown function at location $x_i$; $\sigma$ indicates the noise level in the observed data. 
 The random noise vector $\{e_i\}_{i=1}^n$ consists of independent and identically distributed random variables. We will discuss two different types of random variables and investigate the probabilistic convergence estimates between the reconstructions and exact solution in two cases: convergence respect to expectation and convergence with an exponential tail.   

The rest of the paper is organized as follows. In section \ref{section2}, we prove the uniqueness and stability of the inverse problem \ref{IP-continuous} by constructing a monotone fixed point iteration. Then, in section \ref{section3-Stochastic}, we prove the stochastic error estimates for the inverse problem \ref{IP-numeric}. We present a quantitative understanding of the convergence in probability of the regularized solutions. We follow the idea in section \ref{section3-Stochastic} and design the inversion algorithms in section \ref{section4}. We present illustrative two-dimensional numerical results to verify the effectiveness of proposed algorithms and the theoretical convergence estimates.

\section{Uniqueness and stability of inverse problem \ref{IP-continuous}.}\label{section2}
 This section aims to investigate the uniqueness and stability of the nonlinear inverse source problem \ref{IP-continuous}. Our approach is to propose a monotone operator which generates a pointwise decreasing sequence converging to the exact source $q^*$. 

\subsection{Positivity results.}

Firstly, in the next two lemmas we recall the regularity properties and maximum principle \cite{Evans:1998}.

\begin{lemma}\label{lemma_regularity}
The parabolic model of $v$ is given as  
 \begin{equation*}
\begin{cases}
\begin{aligned}
(\partial_t-\Delta) v(x,t)+p(x)v(x,t) &=f(x,t), && (x,t) \in \Omega\times(0,T],\\
\B v(x,t) &= 0, &&(x,t)\in\partial\Omega\times(0,T],\\
v(x,0)&=0,&&x\in\Omega.
\end{aligned}
\end{cases}
\end{equation*}
With $f\in H^1(0,T;L^2(\Omega))$, $v$ satisfies the next regularity result 
\begin{equation*}
 \max_{t\in[0,T]} (\|v(\cdot,t)\|_{H^2(\Omega)}+\|\partial_t v(\cdot,t)\|_{L^2(\Omega)})\le C\|f\|_{H^1(0,T;L^2(\Omega))},
\end{equation*}
where the constant $C$ depends on $\Omega$, $T$ and $p$. We could see that the $H^2$ regularity and the fact $1\le d\le 3$ ensure the continuity of $v(\cdot, t)$.  
\end{lemma}

\begin{lemma}\label{lemma_maximum}
We give the next parabolic model  
 \begin{equation*}
\begin{cases}
\begin{aligned}
(\partial_t-\Delta) v(x,t)+\tilde p(x)v(x,t) &=f(x,t), && (x,t) \in \Omega\times(0,T],\\
\B v(x,t) &= \tilde b(x,t), &&(x,t)\in\partial\Omega\times(0,T],\\
v(x,0)&=v_0(x),&&x\in\Omega.
\end{aligned}
\end{cases}
\end{equation*}
Let $f$, $\tilde p$, $\tilde b$ and $v_0$ be continuous and nonnegative. Then we have $v\ge 0$ a.e. on $\Omega\times(0,T]$. Also, if $v$ attains a nonpositive minimum over $\overline{\Omega}\times[0,T]$ on a point $(x_0,t_0)\in \Omega\times(0,T]$, then $v|_{\Omega\times(0,t_0]}$ is constant. 
\end{lemma}

Throughout the paper, we need the following assumptions to be valid. 

\begin{assumption}\label{assumption}
The boundary and initial conditions $u_0,\ b$, the potential $p$, and the exact source $q^*$ are continuous and satisfy the following conditions. 
 \begin{itemize}
  \item [(a)] $b$, $\partial_t b$ and $\partial^2_t b$ are nonnegative on $\partial \Omega\times(0,T]$, and $b$ can not be vanishing.  Also we denote the maximum of $b$, $\partial_t b$ and $\partial^2_t b$ on $\partial \Omega\times(0,T]$ by $M_b$. 
  \item [(b)] $b(x,T)>0$ on $\partial\Omega$.
   \item [(c)] $p$ is positive and satisfies $p\ge \Delta g/g$ in $\Omega$. 
   \item [(d)] $q^*\in \mathcal{Q}:=\{q\in C(\bar\Omega): 0\leq q \leq M\}$, where $M>0$ is the upper bound of the admissible set $\mathcal Q$.
 \end{itemize}
\end{assumption}

With Assumption \ref{assumption} and Lemma \ref{lemma_maximum}, we could deduce the positivity results of the solutions of equations \ref{PDE_ue} and \ref{PDE_um}.

\begin{lemma}\label{lemma_lower}
Let $v_M$ satisfy the next model: 
\begin{equation*}
\begin{cases}
\begin{aligned}
(\partial_t-\Delta) v_M(x,t)+(p(x)+M)v_M(x,t) &=0, && x \in \Omega\times(0,T],\\
\B v_M(x,t) &= b(x,t), &&x\in\partial\Omega\times(0,T],\\
v_M(x,0)&=0, &&x\in\Omega,\\
\end{aligned}
\end{cases}
\end{equation*}
where $M$ is given by Assumption \ref{assumption}. Defining $m_{\mathcal Q}:=\min_{x\in\Omega}v_M(x,T)$, we have $m_{\mathcal Q}>0$. 
\end{lemma}
\begin{proof}
From Assumption \ref{assumption} and Lemma \ref{lemma_maximum}, we have $v_M(x,T)>0$ on $\Omega$. Recalling that $b(x,T)>0$ and the continuity of $v_M(\cdot,T)$, which come from Assumption \ref{assumption} and Lemma \ref{lemma_regularity}, respectively, we can deduce that $\min_{x\in\Omega} v_M(x,T)>0$. The proof is complete. 
\end{proof}

\begin{lemma}\label{lemma_positivity}
With Assumption \ref{assumption}, the $u_m$ and $u_e$ defined in models \ref{PDE_ue} and \ref{PDE_um} satisfy the following results:
\begin{itemize}
\item [(a)]$u_e\ge 0$ on $\Omega\times(0,T]$, and 
 $u_e(x,T)\ge m_{\mathcal Q}>0$ on $\Omega$.
\item [(b)]$u_m(x,T)>0$ on $\Omega$.
\item [(c)]$\partial_t u_e$ and $\partial_t^2 u_e$ are nonnegative on $\Omega\times(0,T]$.
\item [(d)]$\partial_t u_m\ge 0$ on $\Omega\times(0,T]$.
\item [(e)] $\max\limits_{(x,t)\in\Omega\times(0,T]}\{u_e,\partial_t u_e,\partial^2_t u_e\}\le M_b$.
\end{itemize}
The constants $m_{\mathcal Q}$ and $M_b$ are given by Lemma \ref{lemma_lower} and Assumption \ref{assumption}, respectively.
\end{lemma}
\begin{proof}
 For statement $(a)$, obviously $u_e\ge 0$ on $\Omega\times(0,T]$ from Lemma \ref{lemma_maximum}. 
 We set $v=u_e-v_M$, where $u_e$ and $v_M$ are introduced in equation \ref{PDE_ue} and Lemma \ref{lemma_lower}, respectively. Then we could deduce the model for $v$ as 
 \begin{equation*}
\begin{cases}
\begin{aligned}
(\partial_t-\Delta) v(x,t)+(p+q^*)v(x,t) &=(M-q^*)v_M\ge 0, && x \in \Omega\times(0,T],\\
\B v(x,t) &=0, &&x\in\partial\Omega\times(0,T],\\
v(x,0)&=0, &&x\in\Omega,\\
\end{aligned}
\end{cases}
\end{equation*}
where the non-negativity of $v_M$ follows from Lemma \ref{lemma_maximum} straightforwardly. Using Lemma \ref{lemma_maximum} again, we have $v=u_e-v_M\ge 0$, which gives $u_e(x,T)\ge v_M(x,T)\ge  m_{\mathcal Q}>0$.

For statement $(b)$, Lemma \ref{lemma_maximum} leads to $u_m(x,T)\ge 0$ on $\Omega$. Assume that $u_m(x_0,T)=0$ with $x_0\in\Omega$. Then Lemma \ref{lemma_maximum} yields that 
$u_m\equiv 0$, which gives $u_e\equiv 0$. This is a contradiction. Hence, $u_m(x,T)>0$ on $\Omega$. 
 
For statement $(c)$, setting $v=\partial_t u_e$, it satisfies 
 \begin{equation*}
\begin{cases}
\begin{aligned}
(\partial_t-\Delta) v+(p(x)+q^*(x))v &=0, && (x,t) \in \Omega\times(0,T],\\
\B v(x,t) &=\partial_tb, && (x,t)\in\partial\Omega\times(0,T],\\
v(x,0)&=0,&&x\in\Omega.
\end{aligned}
\end{cases}
\end{equation*}
 From Assumption \ref{assumption} and Lemma \ref{lemma_maximum}, we have the nonnegativity of $v$. The above arguments could also give the nonnegativity of $\partial^2_t u_e$. 
 
 The proof of statement $(d)$ follows from the one of statement $(c)$. 
 
 For statement $(e)$, setting $v=M_b-u_e$, it satisfies 
 \begin{equation*}
\begin{cases}
\begin{aligned}
(\partial_t-\Delta) v+(p(x)+q^*(x))v &=(p(x)+q^*(x))M_b\ge 0, && x \in \Omega\times(0,T],\\
\B v(x,t) &= M_b-b(x,t)\ge 0, &&x\in\partial\Omega\times(0,T],\\
v(x,0)&=M_b\ge 0, &&x\in\Omega.
\end{aligned}
\end{cases}
\end{equation*}
We have that $v=M_b-u_e\ge 0$ on $\Omega\times(0,T]$ immediately from Lemma \ref{lemma_maximum}. 
Similarly, if we set $v=M_b-\partial_t u_e$, then it holds that 
\begin{equation*}
\begin{cases}
\begin{aligned}
(\partial_t-\Delta) v+(p(x)+q^*(x))v &=(p(x)+q^*(x))M_b\ge 0, && x \in \Omega\times(0,T],\\
\B v(x,t) &= M_b-\partial_tb\ge 0, &&x\in\partial\Omega\times(0,T],\\
v(x,0)&=M_b\ge 0, &&x\in\Omega.
\end{aligned}
\end{cases}
\end{equation*}
Lemma \ref{lemma_maximum} leads to $M_b-\partial_t u_e\ge 0$ on $\Omega\times(0,T]$. At last, the model of $v=M_b-\partial ^2_tu_e$ would be given as 
\begin{equation*}
\begin{cases}
\begin{aligned}
(\partial_t-\Delta) v+(p(x)+q^*(x))v &=(p(x)+q^*(x))M_b\ge 0, && x \in \Omega\times(0,T],\\
\B v(x,t) &= M_b-\partial_t^2b\ge 0, &&x\in\partial\Omega\times(0,T],\\
v(x,0)&=M_b\ge 0, &&x\in\Omega.
\end{aligned}
\end{cases}
\end{equation*}
The similar arguments lead to $M_b-\partial ^2_tu_e\ge 0$ on $\Omega\times(0,T]$.

 The proof is complete.   
\end{proof}

\subsection{Operator $K$ and the monotonicity.}
The reconstruction of the inverse problem relies on the operator $K$, which is defined as 
\begin{equation}\label{operator}
 K q=\frac{\partial_t u_m(x,T;q)-\Delta g(x)+p(x)g(x)}{u_e(x,T;q)},
\end{equation}
with domain 
\begin{equation}\label{domain}
 \mathcal D=\{q\in C(\Omega): M\ge q(x)\ge (-\Delta g(x)+p(x)g(x))/u_e(x,T;0) \}.
\end{equation}
\begin{remark}
 Lemma \ref{lemma_positivity} gives the positivity of $u_e(x,T;q)$, which ensures the well-definedness of operator $K$. The constant $M$ in \ref{domain} comes from Assumption \ref{assumption}. We can set it sufficiently large to support the well-definedness of $\mathcal D$.  
 \end{remark}

 The next lemma contains the equivalence between the fixed point of $K$ and the solution of the inverse problem. 
\begin{lemma}\label{lemma_equivalence}
 Recalling the equation \ref{PDE_um} for $u_m$ and the operator $K$ in \ref{operator}, the next two statements are equivalent:
\begin{itemize}
\item [(i)] $q$ is one fixed point of $K$;
\item [(ii)] $q$ satisfies $u_m(x,T;q)=g(x)$. 
\end{itemize}
\end{lemma}
\begin{proof}
 Assuming $u_m(x,T;q)=g(x)$, it is not hard to see that $q$ is one fixed point of $K$.  
 
 Let us prove conversely. If $q$ is the fixed point, then we have 
 \begin{equation*}
\begin{cases}
\begin{aligned}
-\Delta v(x)+p(x)v(x) &=0, && x\in \Omega,\\
\B v(x) &=0, &&x\in\partial\Omega,
\end{aligned}
\end{cases}
\end{equation*}
where $v(x)=g(x)-u_m(x,T;q)$. We could see that the solution of the above model should be vanishing, i.e.  $u_m(x,T;q)=g(x).$ The proof is complete. 
\end{proof}

The next lemma is the monotonicity of operator $K$. 
\begin{lemma}\label{lemma_monotone}
 Given $q_1,q_2\in \mathcal D$ with $q_1\leq q_2$, we have $Kq_1\le Kq_2.$
\end{lemma}
\begin{proof}
 From the definition \ref{operator} of $K$, we have 
 \begin{align*}
  &Kq_1-Kq_2\\
  &=u_e^{-1}(x,T;q_1)(\partial_tu_m(x,T;q_1)-\partial_tu_m(x,T;q_2))\\
  &\quad +u_e^{-1}(x,T;q_1)u_e^{-1}(x,T;q_2)(\partial_tu_m(x,T;q_2)-\Delta g+pg)(u_e(x,T;q_2)-u_e(x,T;q_1))\\
  &=:I_1+I_2.
 \end{align*}
 
For $I_2$, we let $v=u_e(x,t;q_1)-u_e(x,t;q_2)$. From model \ref{PDE_ue}, we see that 
\begin{equation}\label{equation1}
\begin{cases}
\begin{aligned}
(\partial_t-\Delta) v+(p+q_2)v &=(q_2-q_1)u_e(x,t;q_1)\ge0, && (x,t) \in \Omega\times(0,T],\\
\B v &= 0, &&(x,t)\in\partial\Omega\times(0,T],\\
v(x,0)&=0, &&x\in\Omega.
\end{aligned}
\end{cases}
\end{equation}
Applying Lemma \ref{lemma_positivity} on the above model leads to $v(x)\ge 0$. Also, Lemma \ref{lemma_positivity} and Assumption \ref{assumption} give that 
$$\partial_tu_m(x,T;q_2)-\Delta g+pg\ge 0,\ u_e(x,T;q_1)>0,\ u_e(x,T;q_2)>0.$$ 
Hence, we prove that $I_2 \le 0$.  

For $I_1$, if we set $w=u_m(x,t;q_1)-u_m(x,t;q_2)$, then it satisfies 
\begin{equation}\label{um1-um2}
\begin{cases}
\begin{aligned}
(\partial_t-\Delta) w+p(x)w &=q_1u_e(x,t;q_1)-q_2u_e(x,t;q_2), && (x,t) \in \Omega\times(0,T],\\
\B w &=0, &&(x,t)\in\partial\Omega\times(0,T],\\
w(x,0)&=0,&&x\in\Omega.
\end{aligned}
\end{cases}
\end{equation}
Actually, the model \ref{equation1} leads to 
\begin{equation}\label{ue1-ue2}
\begin{cases}
\begin{aligned}
(\partial_t-\Delta) v+pv &=q_2u_e(x,t;q_2)-q_1u_e(x,t;q_1), && (x,t) \in \Omega\times(0,T],\\
\B v &= 0, &&(x,t)\in\partial\Omega\times(0,T],\\
v(x,0)&=0, &&x\in\Omega.
\end{aligned}
\end{cases}
\end{equation}
So we see that $w=-v$. If we set $z=\partial_t v$, then we have 
\begin{equation}\label{derivative-ue1-ue2}
\begin{cases}
\begin{aligned}
(\partial_t-\Delta)z+(p+q_2)z &=(q_2-q_1)\partial_tu_e(x,t;q_1)\ge0, && (x,t) \in \Omega\times(0,T],\\
\B z &= 0, &&(x,t)\in\partial\Omega\times(0,T],\\
z(x,0)&=0,&&x\in\Omega,
\end{aligned}
\end{cases}
\end{equation}
which together with Lemma \ref{lemma_positivity} gives $z(x,t)\ge 0$. So we have $\partial_t w(x,T)\le 0$, which leads to $I_1\le 0$. 

Now we have proved $Kq_1-Kq_2\le 0$. The proof is complete. 
\end{proof}

\subsection{Iteration and uniqueness of inverse problem.}
To prove the uniqueness, we firstly give the next two lemmas.
\begin{lemma}\label{lemma_uniqueness1}
 If $q_1,q_2\in\mathcal D$ are both fixed points of $K$ with $q_1\le q_2$, then 
 $q_1=q_2$.  
\end{lemma}
\begin{proof}
 With Lemma \ref{lemma_equivalence}, we have $u_m(x,T;q_1)=u_m(x,T;q_2)=g(x)$. The proof of Lemma \ref{lemma_monotone} gives that $u_m(x,t;q_1)-u_m(x,t;q_2)=u_e(x,t;q_2)-u_e(x,t;q_1)$. Hence, $u_e(x,T;q_2)-u_e(x,T;q_1)=0$. Setting $v=u_e(x,t;q_1)-u_e(x,t;q_2)$ and recalling the model \ref{equation1}, we have 
 \begin{equation*}
 \begin{cases}
\begin{aligned}
(\partial_t-\Delta)v+(p+q_2)v &=(q_2-q_1)u_e(x,t;q_1)\ge0, && (x,t) \in \Omega\times(0,T],\\
\B v &= 0, &&(x,t)\in\partial\Omega\times(0,T],\\
v(x,0)&=v(x,T)=0,&&x\in\Omega.
\end{aligned}
\end{cases}
\end{equation*}
 In the proof of Lemma \ref{lemma_monotone} we have proved that $z=\partial_t v\ge 0$. The results $v(x,0)=v(x,T)=0$ and $\partial_t v\ge 0$ give that $v=0$ a.e. on $\Omega\times[0,T]$. So we have $(q_2-q_1)u_e(x,t;q_1)=0$. With the fact that $u_e(x,T;q_1)>0$ on $\Omega$, we can deduce that $q_1=q_2$ and complete the proof.  
\end{proof}

\begin{lemma}\label{lemma_operator_stability}
Given $q_1,q_2\in \mathcal D$, we have 
\begin{equation*}
 \|Kq_1-Kq_2\|_{L^2(\Omega)}\le C\|q_1-q_2\|_{L^2(\Omega)}.
\end{equation*}
\end{lemma}
\begin{proof}
 From the proof of Lemma \ref{lemma_monotone}, we have $Kq_1-Kq_2=I_1+I_2$, where 
 \begin{align*}
  I_1&=\frac{\partial_tu_m(x,T;q_1)-\partial_tu_m(x,T;q_2)}{u_e(x,T;q_1)},\\
  I_2&=\frac{(\partial_tu_m(x,T;q_2)-\Delta g+pg)(u_e(x,T;q_2)-u_e(x,T;q_1))}{u_e(x,T;q_1)u_e(x,T;q_2)}.
 \end{align*}
Lemma \ref{lemma_lower} gives the positive lower bound of $u_e(x,T;q_1)$ and $u_e(x,T;q_2)$, which is independent of the choice of $q_1$ and $q_2$. Hence we could estimate $I_1$ and $I_2$ as follows. 

For $I_2$, from Lemma \ref{lemma_regularity} and model \ref{equation1} we have 
\begin{equation*}
 \|I_2\|_{L^2(\Omega)}\le C\|(q_1-q_2)u_e(x,t;q_1)\|_{H^1(0,T;L^2(\Omega))}.
\end{equation*}
From Lemma \ref{lemma_positivity} $(e)$, we see that 
$$\|(q_1-q_2)u_e(x,t;q_1)\|_{H^1(0,T;L^2(\Omega))}\le C\|q_1-q_2\|_{L^2(\Omega)},$$ 
where the constant $C$ depends on $M_b$ and $T$. So $\|I_2\|_{L^2(\Omega)}\le C\|q_1-q_2\|_{L^2(\Omega)}$.

For $I_1$, we have that 
\begin{equation*}
 \|I_1\|_{L^2(\Omega)}\le C\|\partial_tu_m(x,T;q_1)-\partial_tu_m(x,T;q_2)\|_{L^2(\Omega)}.
\end{equation*}
From the proof of Lemma \ref{lemma_monotone}, we have 
$$u_e(x,t;q_1)-u_e(x,t;q_2)=u_m(x,t;q_2)-u_m(x,t;q_1).$$
Hence, with Lemma \ref{lemma_regularity} and model \ref{equation1}, it holds that 
\begin{align*}
 \|\partial_tu_m(x,T;q_1)-\partial_tu_m(x,T;q_2)\|_{L^2(\Omega)}
 &= \|\partial_tu_e(x,T;q_1)-\partial_tu_e(x,T;q_2)\|_{L^2(\Omega)}\\
 &\le C\|(q_1-q_2)u_e(x,t;q_1)\|_{H^1(0,T;L^2(\Omega))}\\
 &\le C\|q_1-q_2\|_{L^2(\Omega)}.
\end{align*}
Here the proof for $I_2$ is used. So $\|I_1\|_{L^2(\Omega)}\le C\|q_1-q_2\|_{L^2(\Omega)}$.

Combining the estimates for $I_1$ and $I_2$, we obtain the desired result and complete the proof. 
\end{proof}

Now we could state the uniqueness theorem, whose proof relies on the following iteration
\begin{equation}\label{fix-iteration}
q_0=\frac{-\Delta g(x)+p(x)g(x)}{u_e(x,T;0)},\ q_{n+1}=K q_n, \ n=0,1,2\cdots.
\end{equation}

\begin{theorem}\label{theorem_uniqueness}
Recall that the definitions of operator $K$, domain $\mathcal D$ and sequence 
$\{q_n\}_{n=0}^\infty$ are given in \ref{operator}, \ref{domain} and \ref{fix-iteration}, respectively. 
Under Assumption \ref{assumption}, for a fixed point $q$ of $K$ in domain $\mathcal D$, the sequence $\{q_n\}_{n=0}^\infty$ would  converge to $q$ increasingly. 

This leads to the uniqueness of the inverse problem. More precisely, if $q,\tilde q\in\mathcal D$ are fixed points of $K$, then $q=\tilde q$. 
\end{theorem}

\begin{proof}
Assumption \ref{assumption} and Lemma \ref{lemma_positivity} give that $q_0\ge0$. Setting $v=u_e(x,t;q_0)-u_e(x,t;0)$, we have 
 \begin{equation*}
\begin{cases}
\begin{aligned}
(\partial_t-\Delta) v(x,t)+p(x)v(x,t)&=-q_0u_e(x,t;q_0)\le 0 , && (x,t) \in \Omega\times(0,T],\\
\B v(x,t) &= 0, &&(x,t)\in\partial\Omega\times(0,T],\\
v(x,0)&=0,&&x\in\Omega.
\end{aligned}
\end{cases}
\end{equation*}
With Lemma \ref{lemma_positivity}, we have $v\le 0$, which gives $0< u_e(x,T;q_0)\le u_e(x,T;0)$. Also, Lemma \ref{lemma_positivity} leads to $\partial_t u_m(x,t;q_0)\ge 0$. Hence, we have $q_1=Kq_0\ge q_0$. The fixed point $q\in \mathcal D$ gives that $q\ge q_0$. This together with Lemma \ref{lemma_monotone} leads to $Kq=q\ge Kq_0=q_1$. Hence we have $M\ge q\ge q_1\ge q_0$, and sequentially $q_1\in\mathcal D$ (the continuity of $q_1$ follows from Lemma \ref{lemma_regularity}). Applying Lemma \ref{lemma_monotone}, we obtain that 
$$q=Kq\ge q_2=Kq_1\ge Kq_0=q_1.$$ 
Continuing the above procedure, we conclude that $\{q_n\}_{n=0}^\infty$ is an increasing sequence satisfying  $q_n\le q,\ \N+$.

Now we have proved that $\{q_n\}_{n=0}^\infty$ is an increasing sequence with upper bound $q$. This means the pointwise convergence of $\{q_n\}_{n=0}^\infty$ and we denote the limit by $q^\dag$. From the Monotone Convergence Theorem, we have $\|q^\dag-q_n\|_{L^2(\Omega)}\to 0$. Also, it is obvious that $q^\dag\le q$. 

Next we need to show that $q^\dag$ is a fixed point of $K$ in $\mathcal D$. From the triangle inequality, we have 
\begin{equation*}
 \|q^\dag-Kq^\dag\|_{L^2(\Omega)}\le \|q^\dag-Kq_n\|_{L^2(\Omega)}+\|Kq_n-Kq^\dag\|_{L^2(\Omega)}=:I_3+I_4.
\end{equation*}
For $I_3$, we see that 
$$I_3= \|q^\dag-Kq_n\|_{L^2(\Omega)}= \|q^\dag-q_{n+1}\|_{L^2(\Omega)}\to 0,\ n\to\infty.$$
For $I_4$, Lemma \ref{lemma_operator_stability} gives that 
$$I_4=\|Kq_n-Kq^\dag\|_{L^2(\Omega)}\le C\|q_n-q^\dag\|_{L^2(\Omega)}\to 0,\ n\to\infty.$$
So we have $\|q^\dag-Kq^\dag\|_{L^2(\Omega)}=0$, which means $q^\dag$ is a fixed point of $K$. 
We have proved that $q^\dag$ and $q$ are fixed points of $K$ in $\mathcal D$ with $q^\dag\le q$. With Lemma \ref{lemma_uniqueness1}, we conclude that $q^\dag=q$. 

Next we show the uniqueness of fixed point of $K$ in $\mathcal{D}$. Suppose there are two fixed points $q$ and $\tilde q$ of $K$, the above arguments give that $q_n\to q$ and $q_n\to \tilde q$, which leads to $q=\tilde q$ and completes the proof. 
\end{proof}

\subsection{Stability of inverse problem.}
In this section, we will introduce the stability of inverse problem \ref{IP-continuous} under some conditions. Firstly, we show the energy estimates for the solutions below.
\begin{lemma}\label{lem-energy}
Define $C_p:={\rm Inf} \, p(x)$ and Assumption \ref{assumption} yields that $C_p>0$; suppose that ${\hat q},\, {\tilde q}\in\mathcal{D}$, where $\mathcal D$ is given by \ref{domain}. For equation \ref{PDE_ue}, we have the following estimates
\begin{equation}\label{energyestimate}
\begin{aligned}
\max\big\{\|u_e(\cdot ,T;{\tilde q})-u_e(\cdot,T;{\hat q})\|_{L^2(\Omega)}, \|\partial_t u_e(\cdot,T;{\tilde q})-\partial_t u_e(\cdot,T;{\hat q})\|_{L^2(\Omega)}\big\} &\le  C\|{\hat q}-{\tilde q}\|_{L^2(\Omega)},\\
\max\big\{\|u_e(\cdot,T;{\tilde q})-u_e(\cdot,T;{\hat q})\|_{(H^1(\Omega))^*}, \|\partial_t u_e(\cdot,T;{\tilde q})-\partial_t u_e(\cdot,T;{\hat q})\|_{(H^1(\Omega))^*}\big\}&\le C \|{\hat q}-{\tilde q}\|_{(H^1(\Omega))^*}.
\end{aligned}
\end{equation}
Here $C:=(\sqrt{T}M_b)/\sqrt{C_p}$ is a positive constant. The above estimates are also valid for solution $u_m$ of equation \ref{PDE_um}.
\end{lemma}
\begin{proof}
Let $U_e:=u_e(x,t;{\tilde q})-u_e(x,t;{\hat q})$. From \ref{equation1}, we see that
\begin{equation}\label{equation-U}
\begin{cases}
\begin{aligned}
(\partial_t-\Delta) U_e+(p+{\hat q})U_e &=({\hat q}-{\tilde q})u_e(x,t;{\tilde q}), && (x,t) \in \Omega\times(0,T),\\
\B U_e &= 0, &&(x,t)\in\partial\Omega\times(0,T),\\
U_e(x,0)&=0, &&x\in\Omega.
\end{aligned}
\end{cases}
\end{equation}
Then, by the energy estimate for the solution of the initial boundary value problem for the diffusion equation and Young's inequality, we have
\begin{align*}
\frac{1}{2}\|U_e(\cdot,T)\|_{L^2(\Omega)}^2+&\beta \|U_e\|^2_{L^2(\partial\Omega\times(0,T))}
+\int_0^T\int_\Omega |\nabla U_e|^2dxdt\\
+&\int_0^T\int_\Omega (p(x)+{\hat q}(x))|U_e(x,t)|^2dxdt \\
=&\int_0^T\int_\Omega ({\hat q}(x)-{\tilde q}(x)) u_e(x,t;{\tilde q})U_e(x,t)dx\;dt
\nonumber\\
\le&\frac{1}{2}{C_p} \int_0^T\int_\Omega|U_e(x,t)|^2dx\;dt+\frac{1}{2 C_p}
\int_0^T\int_\Omega ({\hat q}(x)-{\tilde q}(x))^2 \,|u_e(x,t;{\tilde q})|^2\,dx\;dt,
\end{align*}
Then, noticing that $\beta>0$, ${\hat q}(x)\ge0$, $p(x)>0$ and $C_p:={\rm Inf} \, p(x)$, we have
$$
\frac{1}{2}\|U_e(\cdot,T)\|_{L^2(\Omega)}^2 \leq \frac{1}{2 C_p}
\int_0^T\int_\Omega ({\hat q}(x)-{\tilde q}(x))^2 \,|u_e(x,t;{\tilde q})|^2\,dx\;dt,
$$
which leads to $\|u_e(x,T;{\tilde q})-u_e(x,T;{\hat q})\|_{L^2(\Omega)}\leq \frac{\sqrt{T}M_b}{\sqrt{C_p}} \|{\hat q}-{\tilde q}\|_{L^2(\Omega)}$. As shown in \ref{um1-um2}--\ref{ue1-ue2}, we know $u_m(x,t;{\hat q})-u_m(x,t;{\tilde q})=u_e(x,t;{\tilde q})-u_e(x,t;{\hat q})$. Hence it holds also that 
$$\|u_m(x,T;{\tilde q})-u_m(x,T;{\hat q})\|_{L^2(\Omega)}\leq \frac{\sqrt{T}M_b}{\sqrt{C_p}} \|{\hat q}-{\tilde q}\|_{L^2(\Omega)}.$$ 

Next, setting $Z=\partial_t u_e(x,t;{\tilde q})-\partial_t u_e(x,t;{\hat q})=\partial_t u_m(x,t;{\hat q})-\partial_t u_m(x,t;{\tilde q})$, we have
\begin{equation}\label{derivative-um1-um2}
\begin{cases}
\begin{aligned}
(\partial_t-\Delta)Z+(p+{\hat q})Z &=({\tilde q}-{\hat q})\partial_tu_e(x,t;{\tilde q}), && (x,t) \in \Omega\times(0,T),\\
\B Z &= 0, &&(x,t)\in\partial\Omega\times(0,T),\\
Z(x,0)&=0,&&x\in\Omega.
\end{aligned}
\end{cases}
\end{equation}
Again, by the energy estimate for the solution to \ref{derivative-um1-um2}, we have
\begin{align*}
\frac{1}{2}\|Z(\cdot,T)\|_{L^2(\Omega)}^2+&\beta \|Z\|^2_{L^2(\partial\Omega\times(0,T))}
+\int_0^T\int_\Omega |\nabla Z|^2dxdt\\
+&\int_0^T\int_\Omega (p(x)+{\hat q}(x))|Z(x,t)|^2dxdt \\
=&\int_0^T\int_\Omega ({\tilde q}(x)-{\hat q}(x)\partial_t u_e(x,t;{\tilde q})Z(x,t)dx\;dt
\nonumber\\
\le&\frac{1}{2}{C_p} \int_0^T\int_\Omega|Z(x,t)|^2dx\;dt+\frac{1}{2 C_p}
\int_0^T\int_\Omega ({\tilde q}(x)-{\hat q}(x))^2 \,|\partial_t u_e(x,t;{\tilde q})|^2\,dx\;dt.
\end{align*}
Then we have the $L^2$ estimate in \ref{energyestimate}, and the $(H^1)^*$ estimate could be deduced analogously. The proof is complete.
\end{proof}

Now it is time to show the stability of inverse problem \ref{IP-continuous}, which is concluded as the following theorem.
\begin{theorem}\label{thm-stability}
Let Assumption \ref{assumption} be valid. Suppose ${\hat q},\, {\tilde q}\in\mathcal{D}$ and $C_p:={\rm Inf} \, p(x)$ is large enough such that 
$$
\frac{\sqrt{T} M_b (M+1)}{m_{\mathcal Q}\sqrt{C_p}}<1,
$$ 
where $M,\ M_b,\ m_{\mathcal Q}$ can be found in Assumption
\ref{assumption} and Lemma \ref{lemma_lower}. Then there holds that
\begin{align*}
\|{\hat q}-{\tilde q}\|_{L^2(\Omega)}
&\leq C \left(\|{\Delta (\mathbb{G}{\hat q}-\mathbb{G}{\tilde q})}\|_{L^2(\Omega)}+\|p\|_{L^\infty(\Omega)}\|\mathbb{G}{\hat q}-\mathbb{G}{\tilde q}\|_{L^2(\Omega)}\right),\\
\|{\hat q}-{\tilde q}\|_{(H^1(\Omega))^*}
&\leq C \left(\|{\Delta (\mathbb{G}{\hat q}-\mathbb{G}{\tilde q})}\|_{(H^1(\Omega))^*}+\|p\|_{L^\infty(\Omega)}\|\mathbb{G}{\hat q}-\mathbb{G}{\tilde q}\|_{(H^1(\Omega))^*}\right),
\end{align*}
where the positive constant $C:=\sqrt{C_p}/(m_{\mathcal Q}\sqrt{C_p}-\sqrt{T} M_b (M+1))$. 
\end{theorem}
\begin{proof}
Under the given conditions, we have
\ben
{\hat q}-{\tilde q} = \frac{\partial_t u_m(x,T;{\hat q})-\Delta \mathbb{G}{\hat q}+p \mathbb{G}{\hat q}}{u_e(x,T;{\hat q})} - \frac{\partial_t u_m(x,T;{\tilde q})-\Delta \mathbb{G}{\tilde q}+p \mathbb{G}{\tilde q}}{u_e(x,T;{\tilde q})} =: J_1+J_2,
\een
where
$$
J_1:=\frac{\partial_tu_m(x,T;{\hat q})-\Delta \mathbb{G}{\hat q}+p\mathbb{G}{\hat q}}{u_e(x,T;{\hat q})}-\frac{\partial_t u_m(x,T;{\hat q})-\Delta \mathbb{G}{\hat q}+p \mathbb{G}{\hat q}}{u_e(x,T;{\tilde q})}
$$
and
$$
J_2:=\frac{\partial_t u_m(x,T;{\hat q})-\Delta \mathbb{G}{\hat q}+p \mathbb{G}{\hat q}}{u_e(x,T;{\tilde q})} - \frac{\partial_t u_m(x,T;{\tilde q})-\Delta \mathbb{G}{\tilde q}+p \mathbb{G}{\tilde q}}{u_e(x,T;{\tilde q})},
$$
respectively. 

Firstly, we make an estimate for $J_1$. Notice that $\partial_t u_m(x,T;{\hat q})-\Delta \mathbb{G}{\hat q}+p\mathbb{G}{\hat q}={\hat q} u_e(x,T;{\hat q})$ in $H^2(\Omega)$, we have
$$
J_1=\frac{{\hat q}\big(u_e(x,T;{\tilde q})-u_e(x,T;{\hat q})\big)}{u_e(x,T;{\hat q})}.
$$

Then, from Lemmas \ref{lemma_positivity} and \ref{lem-energy}, it holds that
\begin{align}
\|J_1\|_{L^2(\Omega)} =\left\|\frac{{\hat q}[u_e(x,T;{\tilde q})-u_e(x,T;{\hat q})]}{u_e(x,T;{\hat q})}\right\|_{L^2(\Omega)} \leq \frac{\sqrt{T} M_b M}{m_{\mathcal Q}\sqrt{C_p}}  \|{\hat q}-{\tilde q}\|_{L^2(\Omega)}.
\end{align}

Next, for $J_2$, it is straightforward to see that
\begin{align*}
\|J_2\|_{L^2(\Omega)} \leq &\left\|\frac{\partial_tu_m(x,T;{\hat q})-\partial_t u_m(x,T;{\tilde q})}{u_e(x,T;{\tilde q})}\right\|_{L^2(\Omega)}
+\left\|\frac{\Delta \mathbb{G}{\hat q}-\Delta \mathbb{G}{\tilde q}}{u_e(x,T;{\tilde q})}\right\|_{L^2(\Omega)}+\left\|\frac{p(\mathbb{G}{\hat q}-\mathbb{G}{\tilde q})}{u_e(x,T;{\tilde q})}\right\|_{L^2(\Omega)}.
\end{align*}
From Lemmas \ref{lemma_positivity} and \ref{lem-energy} again, we know 
\begin{equation*}
\left\|\frac{\partial_t u_m(x,T;{\hat q})-\partial_t u_m(x,T;{\tilde q})}{u_e(x,T;\tilde q)}\right\|_{L^2(\Omega)}\le \frac{\sqrt{T}M_b}{m_{\mathcal Q}\sqrt{C_p}} \|{\hat q}-{\tilde q}\|_{L^2(\Omega)}.
\end{equation*}
With the condition  
$\frac{\sqrt{T} M_b (M+1)}{m_{\mathcal Q}\sqrt{C_p}}<1,$ 
the estimates for $J_1$ and $J_2$ lead to  
\begin{align*}
\left(1-\frac{\sqrt{T} M_b (M+1)}{m_{\mathcal Q}\sqrt{C_p}}\right)\|{\hat q}-{\tilde q}\|_{L^2(\Omega)}
\leq \frac{1}{m_{\mathcal Q}}\|{\Delta (\mathbb{G}{\hat q}-\mathbb{G}{\tilde q})}\|_{L^2(\Omega)}+\frac{\|p\|_{L^\infty(\Omega)}}{m_{\mathcal Q}}\|\mathbb{G}{\hat q}-\mathbb{G}{\tilde q}\|_{L^2(\Omega)}.
\end{align*}
The estimate for $\|{\hat q}-{\tilde q}\|_{(H^1(\Omega))^*}$ could be proved similarly. The proof is complete. 
\end{proof}

\section{Stochastic error estimates for inverse problem \ref{IP-numeric}.}\label{section3-Stochastic}
In this section, we will provide the quantitative estimates for the inverse problem \ref{IP-numeric}, which will be divided into the following two sub-problems {\bf P1} and {\bf P2}. 
\subsection{Two sub-problems of inverse problem \ref{IP-numeric}.}
Let the set of discrete points $\{x_i\}_{i=1}^n$ be scattered but quasi-uniformly distributed in $\Omega$. Suppose the measurement comes with noise and takes the form $g_i^\sigma=g(x_i)+e_i, i=1, 2, \cdots, n$,
where $\{e_i\}^{n}_{i=1}$ are independent and identically distributed random variables with zero mean on a probability space. Since the inverse problem \ref{IP-numeric} is ill-posed and the observation detectors are discrete, we divide \ref{IP-numeric} into following two sub-problems:
\begin{itemize}
\item [{\bf P1.}] recovering $-\Delta g$ and $g$ in $\Omega$ from the discrete noisy data $\{g_i^\sigma\}_{i=1}^n$;
\item [{\bf P2.}] recovering the unknown source $q^*$ from the reconstructed data in {\bf P1}.
\end{itemize}

Firstly, define the operator $S: L^2(\Omega)\to H^2(\Omega)$, where $Sf$ satisfies the following elliptic equations 
\begin{equation}
\label{eqn:elliptic}
\left\{
\begin{aligned}
-\Delta (Sf) &= f \quad &\mbox{in } \; &\Omega, \\
\mathcal{B} (Sf) &= 0 \quad \quad &\mbox{on }  \; &\partial \Omega.
\end{aligned} \right.
\end{equation}
Set $f^*:=-\Delta g$ and $Sf^*=g$. Then, the problem {\bf P1} is transformed to approximate $f^*$ and $Sf^*$ by solving an elliptic optimal control problem:
\beq\label{p1}
f^\sigma=\mathop {\rm arg\,min}\limits_{f\in H^s(\Omega)}\frac{1}{n}\sum\limits_{i=1}^{n} {(Sf(x_i)-g_i^\sigma)^2+\lambda \|f\|_{H^s(\Omega)}^2}, \quad s\in\{0,1\},
\eeq
where $\lambda>0$ is the regularization parameter.

Next, similarly with \ref{fix-iteration}, we solve the problem {\bf P2} by the following iteration:
\begin{equation}\label{fix-iteration-step2}
q_0^\sigma=\frac{f^\sigma+p(x) Sf^\sigma}{u_e(x,T;0)},\quad q_{n+1}^\sigma=K^\sigma q_n^\sigma,\quad n=0,1,2\cdots,
\end{equation}
where the operator $K^\sigma$ is defined by
\begin{equation}\label{operator-noise}
K^\sigma q=\frac{\partial_t u_m(x,T;q)+f^\sigma+p(x)Sf^\sigma}{u_e(x,T;q)}.
\end{equation}
Then, we will investigate the stochastic error estimates for the problems {\bf P1} and {\bf P2}.

Before giving the proofs of the stochastic convergence, we need several definitions and properties.
For any $u,v\in C(\bar\Omega)$ and $y\in \mathbb R^n$, we define 
$$
(y,v)_n=\frac{1}{n}\sum^n_{i=1}y_iv(x_i), \q 
(u,v)_n=\frac{1}{n}\sum^n_{i=1}u(x_i)v(x_i),
$$ 
and the empirical semi-norm $\|u\|_n=(\sum_{i=1}^{n} u^2(x_i)/n)^{1/2}$ for any $u\in C(\bar\Omega)$.

Similarly to \cite{Chen-Zhang2022}, we could show that 
\begin{equation*}
\|u\|^2_{L^2(\Omega)}\leq C(\|u\|^2_n+ n^{-2k/d}\|u\|^2_{H^k(\Omega)}),\ \ \|u\|^2_n\leq C(\|u\|^2_{L^2(\Omega)}+ n^{-2k/d}\|u\|^2_{H^k(\Omega)}).
\end{equation*}
We also have for smooth enough domain $\Omega$, 
$$\|Sf\|_{H^{s+2}(\Omega)}\leq C\|f\|_{H^{s}(\Omega)}.$$

In the following two subsections, we will study the probabilistic convergence in two cases: convergence respect to expectation and convergence with an exponential tail.

\subsection{Stochastic convergence respect to expectation.}

Consider independent and identically distributed random variables $\{e_i\}^n_{i=1}$  with
$\mathbb{E}[e_i]=0$ and $\mathbb{E}[e^2_i]\leq \sigma^2$.

 The following Weyl's law  \cite{Fleckinger} plays a key role in studying the convergence of $f^{\sigma}$ to $f^*$.

\begin{lemma}\label{lem:2.1} Suppose $\Omega$ is a bounded open set of $\mathbb{R}^d$ and $a(x)\in C^0(\bar{\Omega})$, then the eigenvalues problem
\beq
\left\{
\begin{aligned}
-\nabla\cdot (a(x)\nabla \psi) &= \mu\psi \quad&\mbox{in } \Omega, \\
 \psi&= 0 \quad \quad&\mbox{on } \partial \Omega,
\end{aligned} \right.
\eeq
has countable positive eigenvalues $\mu_1\le\mu_2\le\cdots\le\mu_k\le\cdots$. Moreover, there exists constants $C_1,C_2>0$ independent of $k$ such that
$C_1 k^{2/d}\le \mu_k\le C_2k^{2/d},k=1,2,\cdots.$
\end{lemma}

From the above lemma, we conclude that the equivalent eigenvalue problem $u=\mu Su$ has countable positive eigenvalues $C_1 k^{2/d}\le \mu_k\le C_2k^{2/d},k=1,2,\cdots.$ 

Following similar proofs of Lemma 2.1 and 2.2 in \cite{Chen-Zhang2022}, we have the lemmas below.
\begin{lemma}\label{lem:2.2} For any $y\in {\mathbb R}^n$, let $u$ to be the interpolation function of the following problem
\begin{equation}
\min_{u\in H^s(\Omega),Su(x_i)=y_i} \|u\|^2_{H^s(\Omega)},
\end{equation}
then $u\in V_n$, where $V_n$ is a n-dimensional subset of $H^s(\Omega)$.
\end{lemma} 

\begin{lemma}\label{lemma:2.3} Let $V_n$ be defined in Lemma \ref{lem:2.2}. The eigenvalue problem
\beq\label{eigen-n}
(\psi ,v)_{H^s}=\rho (S\psi ,Sv)_n, \forall v\in V_n,
\eeq
has $n$ finite eigenvalues $ \rho_k,\ k=1,2,\cdots, n,$ and all the eigenfunctions consist the orthogonal basis of $V_n$ respect to the norm $\|S\cdot\|_n$. Moreover,
there exist constants $C_3>0$ independent of $k$ such that
$C_3k^{2(2+s)/d}\le \rho_k,\ k=1,2,\cdots, n.$
\end{lemma} 

Then, we have the following convergence in expectation.

\begin{theorem}\label{thm:2.1}
Let $f^{\sigma}\in H^s(\Omega)$ be the unique solution of (\ref{p1}).
Then there exist constants $\lambda_0 > 0$ and $C>0$ such that for any $\lambda \leq \lambda_0$,
\be
& &\label{p5}
\mathbb{E} \big[\|Sf^{\sigma}-Sf^*\|^2_n\big] \leq C \lambda \|f^*\|^2_{H^s(\Omega)} + \frac{C\sigma^2}{n\lambda^{d/2/(2+s)}},\\
& &\label{p6}
\mathbb{E} \big[\|f^{\sigma}\|^2_{H^s(\Omega)}\big] \leq C \|f^*\|^2_{H^s(\Omega)} + \frac{C\sigma^2}{n\lambda^{1+d/2/(2+s)}},\quad s\in\{0,1\}.
\ee
If $n^{(4+2s)/d}\lambda\geq 1$,  for $s=0$, we have
\begin{equation}\label{H-1-1}
\mathbb{E} \big[\|f^{\sigma}-f^*\|^2_{(H^{1}(\Omega))^*}\big] \leq C \lambda^{1/2} \|f^*\|^2_{L^2(\Omega)} + \frac{C\sigma^2}{n\lambda^{d/4+1/2}}.
\end{equation}

For $s=1$, we have
\begin{equation}\label{L2-1-1}
\mathbb{E} \big[\|f^{\sigma}-f^*\|^2_{L^2(\Omega)}\big] \leq C \lambda^{1/3} \|f^*\|^2_{H^1(\Omega)} + \frac{C\sigma^2}{n\lambda^{d/6+2/3}}.
\end{equation}
\end{theorem}

\begin{proof} It is clear that $f^{\sigma}\in L^2(\Omega)$ satisfies the following variational equation
\beq
\label{p7}
\lambda (f^{\sigma},v)+(Sf^{\sigma},Sv)_n =(y,Sv)_n,\ \ \ \forall v\in L^2(\Omega).
\eeq
For any $v\in H^s(\Omega)$, denote the energy norm
$\lla v\rra^2_{\lambda}:=\lambda(v,v)_{H^s}+\|Sv\|_n^2$.
By taking $v=f^{\sigma}-f^*$ in \ref{p7} one obtains easily
\beq\label{x4}
\lj f^{\sigma}-f^*\rj_{\lam}\le \lam^{1/2} |f^*|_{L^2(\Om)}+\sup_{v\in L^2(\Om)}\frac{(e,Sv)_n}{\lla v\rra_{\lam}},
\eeq
where $e$ represents the random error vector.

We will just have to estimate the term $\sup_{v\in {H^s}(\Om)}\frac{(e,Sv)^2_n}{\lla v\rra_{\lam}^2}$. 
\ben
\sup_{v\in {H^s}(\Om)}\frac{(e,Sv)^2_n}{\lla v\rra_{\lam}^2}
&=&\sup_{v\in {H^s}(\Om)}\frac{(e,Sv)^2_n}{\lambda(v,v)_{H^s}+\|Sv\|_n^2}\\
&= & \sup_{v\in {H^s}(\Om)}\frac{(e,Sv)^2_n}{\lambda \min_{u\in {H^s}(\Om), Su(x_i)=Sv(x_i)}(u,u)_{H^s}+\|Sv\|_n^2}\\
&= & \sup_{v\in {H^s}(\Om)}\frac{(e,Sv)^2_n}{\lambda \min_{u\in V_n, Su(x_i)=Sv(x_i)}(u,u)+\|Sv\|_n^2}\\
&= & \sup_{v\in V_n}\frac{(e,Sv)^2_n}{\lambda(v,v)_{H^s}+\|Sv\|_n^2}.\\
\een
In the above equalities, we have applied the Lemma \ref{lem:2.2}.

Let $\rho_1\le\rho_2\le\cdots\le\rho_n$ be the eigenvalues of the problem
\beq\label{x2}
(\psi,v)_{H^s}=\rho(S\psi,Sv)_n\ \ \ \ \forall v\in V_n,
\eeq
and $\{\psi_k\}^n_{k=1}$ be the eigenfunctions of \ref{x2} corresponding to eigenvalues $\{\rho_k\}^n_{k=1}$ satisfying
$(S\psi_k,S\psi_l)_n=\delta_{kl}$, where $\delta_{kl}$ is the Kronerker delta function, $k,l=1,2,\cdots, n$. We see that $\{\psi_k\}^n_{k=1}$ is an orthonormal basis of $V_n$ in the inner product $(S\cdot,S\cdot)_n$. Now for any $v\in V_n$, we have the expansion $v(x)=\sum^n_{k=1}v_k\psi_k(x)$, where $v_k=(Sv,S\psi_k)_n$, $k=1,2,\cdots,n$. Thus
$\lla v\rra^2_{\lam}=\sum^n_{k=1}(\lam\rho_k+1)v_k^2$.
By the Cauchy-Schwarz inequality
\ben
(e,Sv)_n^2 &= & \frac{1}{n^2}\left(\sum^n_{i=1}e_i\left(\sum^n_{k=1} v_k S\psi_k(x_i)\right)\right)^2 = \frac{1}{n^2}\left(\sum^n_{k=1}v_k\left(\sum^n_{i=1} e_i S\psi_k(x_i)\right)\right)^2 \\
&\leq &\frac 1{n^2}\sum^n_{k=1}(1+\lam_n\rho_k)v_k^2\cdot\sum^n_{k=1}(1+\lam_n\rho_k)^{-1}\Big(\sum^n_{i=1}e_i (S\psi_k)(x_i)\Big)^2.
\een
This implies
\ben
\mathbb{E}\left[\sup_{v\in V_n}\frac{(e,v)_n^2}{\lla v\rra^2_{\lam}}\right]
\le\frac 1{n^2}\sum^n_{k=1}(1+\lam\rho_k)^{-1}\mathbb{E}\left(\sum^n_{i=1}e_i S\psi_k(x_i)\right)^2
=\sigma^2n^{-1}\sum^n_{k=1}(1+\lam\rho_k)^{-1},
\een
where we have used the fact that $\|S\psi_k\|_n=1$. Now by Lemma \ref{lemma:2.3} we obtain
\ben
\mathbb{E}\left[\sup_{v\in H^s(\Om)}\frac{(e,Sv)_n^2}{\lla v\rra^2_{\lam}}\right]
&\le&C\sigma^2 n^{-1}\sum_{k=1}^n(1+\lam k^{2(2+s)/d})^{-1}\\
&\le&C\sigma^2 n^{-1}\int^\infty_{1}(1+\lam t^{2(2+s)/d})^{-1}dt
=C\frac{\sigma^2}{n\lam^{d/2/(2+s)}}.
\een
This completes the proof by using \ref{x4}.

Next, with the assumption that $n^{d/2/(2+s)}\lambda \geq 1$, we will have 
\be
\mathbb{E} \big[\|Sf^{\sigma}-S f^*\|^2_{L^2(\Omega)}\big] \leq C \lambda \|f^*\|^2_{L^2(\Omega)} + \frac{C\sigma^2}{n\lambda^{d/2/(2+s)}},
\ee
Apply the interpolation inequality in Sobolev spaces, there exist $\theta_0$ such that for all $\theta\leq \theta_0$
\be
\|Sf^{\sigma}-S f^*\|^2_{H^1(\Omega)} \leq 
C\theta^{-1}\|Sf^{\sigma}-S f^*\|^2_{L^2(\Omega)}+C\theta\|Sf^{\sigma}-S f^*\|^2_{H^2(\Omega)}
\ee
Take $\theta=\lambda^{1/2}$,
\ben
\mathbb{E} \big[\|Sf^{\sigma}-S f^*\|^2_{H^1(\Omega)}\big]
& \leq & C\lambda^{-1/2}\mathbb{E} \big[\|Sf^{\sigma}-S f^*\|^2_{L^2(\Omega)}\big]+C\lambda^{1/2}\mathbb{E} \big[\|Sf^{\sigma}-S f^*\|^2_{H^2(\Omega)}\big]\\
&\leq & C \lambda^{1/2} \|f^*\|^2_{L^2(\Omega)} + \frac{C\sigma^2}{n\lambda^{d/2/(2+s)+1/2}}.
\een

With the definition of the operator $S$, we have for any $v\in H^1(\Omega)$,
\ben
|(f^*-f^{\sigma},v)|=|(a\nabla S(u^*-f^{\sigma}),\nabla v)|
\leq C\|Sf^{\sigma}-S f^*\|^2_{H^1(\Omega)}\|v\|^2_{H^1(\Omega)}
\een
This implies 
\ben
\mathbb{E} \big[\|f^{\sigma}-f^*\|^2_{(H^{1}(\Omega))^*}\big] \leq C\mathbb{E} \big[\|Sf^{\sigma}-S f^*\|^2_{H^1(\Omega)}\big]
\leq C \lambda^{1/2} \|f^*\|^2_{L^2(\Omega)} + \frac{C\sigma^2}{n\lambda^{d/2/(2+s)+1/2}}.
\een
Moreover, for $s=1$,
\be \label{H2-H3}
\|Sf^{\sigma}-S f^*\|^2_{H^2(\Omega)} \leq 
C\theta^{-2}\|Sf^{\sigma}-S f^*\|^2_{L^2(\Omega)}+C\theta\|Sf^{\sigma}-S f^*\|^2_{H^3(\Omega)}
\ee
Take $\theta=\lambda^{1/3}$,
\ben
\mathbb{E} \big[\|Sf^{\sigma}-S f^*\|^2_{H^2(\Omega)}\big] 
&\leq & C\lambda^{-2/3}\mathbb{E} \big[\|Sf^{\sigma}-S f^*\|^2_{L^2(\Omega)}\big]+C\lambda^{1/3}\mathbb{E} \big[\|Sf^{\sigma}-S f^*\|^2_{H^3(\Omega)}\big]\\
&\leq & C \lambda^{1/3}\|f^*\|^2_{H^1(\Omega)} + \frac{C\sigma^2}{n\lambda^{d/2/(2+s)+2/3}}.
\een
This gives the estimate \ref{L2-1-1} and the proof is complete.
\end{proof}

Balancing all the two terms on the right side \ref{p5}-\ref{L2-1-1}, Theorem \ref{lem:2.1} suggests that an proper choice of the parameter $\lam$ is such that $$\lam^{1+d/2/(2+s)}=O(\sigma^2n^{-1}\|f^*\|_{H^s(\Om)}^{-2}).$$

Now we are ready to give the main result of this subsection. Combining Theorem \ref{thm-stability} and Theorem \ref{thm:2.1}, we have the following convergence estimates.

\begin{theorem}\label{thm-estimate-3.2}
Let $q^*$ be the fixed point of iteration \ref{fix-iteration} and $q^\sigma$ be the fixed point of iteration \ref{fix-iteration-step2}, respectively. Then, under the conditions in Theorem \ref{thm-stability}, there exists a constant $C>0$ depending only on $C_p$, $M_b$, $M$, $m_{\mathcal Q}$ and $T$ such that:
\begin{itemize}
\item [(i)] For the case of $s=0$ in \ref{p1}, it holds
$$
\mathbb{E} \big[\|q^{\sigma}-q^*\|_{(H^1(\Omega))^*}\big] \leq C (\lambda^{1/4}+\lambda^{1/2} \|p\|_{L^\infty(\Omega)} )\|f^*\|^2_{L^2(\Omega)},
$$
where $\lam^{1/2+d/8}=O(\sigma n^{-1/2}\|f^*\|_{L^2(\Om)}^{-1})$;

\item [(ii)]  For the case of $s=1$ in \ref{p1}, it holds
$$
\mathbb{E} \big[\|q^{\sigma}-q^*\|_{L^2(\Omega)}\big] \leq C (\lambda^{1/6}+\lambda^{1/2}\|p\|_{L^\infty(\Omega)} )\|f^*\|^2_{H^1(\Omega)},
$$
where
$\lam^{1/2+d/12}=O(\sigma n^{-1/2}\|f^*\|_{H^1(\Om)}^{-1})$. 
\end{itemize} 
\end{theorem}

\subsection{Probabilistic convergence with an exponential tail.}

In this subsection we assume the noises $e_i$, $i=1,2,\cdots,n$, are
independent and identically distributed sub-Gaussian random variables with parameter $\sigma>0$. A random variable $X$ is sub-Gaussian with parameter $\sigma$ if it satisfies
\beq\label{e1}
\mathbb{E}\left[e^{\lam(X-\mathbb{E}[X])}\right]\le e^{\frac 12\sigma^2\lam^2}\ \ \ \ \forall\lam\in\R.
\eeq
The probability distribution function of a sub-Gaussian random variable has a exponentially decaying tail, that is, if $X$ is a sub-Gaussian random variable, then
\beq\label{gg1}
\mathbb{P}(|X-\E [X]|\ge z)\le 2e^{-\frac 12 z^2/\sigma^2}\ \ \forall z>0.
\eeq

We will study the stochastic convergence of the error $\|Sf^*-S f^{\sigma}\|_n$ which characterizes the tail property of $\mathbb{P}(\|Sf^*-S f^{\sigma}\|_n\ge z)$ for $z>0$.

We will use several tools from the theory of empirical processes \cite{Vaart, Geer} for our analysis. The Orilicz norm $\|X\|_\psi$ of a random variable $X$ is defined as
\beq\label{e2}
\|X\|_\psi=\inf\left\{C>0:\mathbb{E}\left[\psi\left(\frac{|X|}C\right)\right]\le 1\right\}.
\eeq
We will use the $\|X\|_{\psi_2}$ norm with
$\psi_2(t)=e^{t^2}-1$ for any $t>0$. By definition,
\beq\label{e3}
\mathbb{P}(|X|\ge z)\le 2\,e^{-z^2/\|X\|_{\psi_2}^2}\ \ \ \ \forall z>0.
\eeq
 The following lemma from \cite[Lemma 2.2.1]{Vaart} is the inverse of \ref{e3}.
\begin{lemma}\label{lem:4.1}
If there exist positive constants $C,K$ such that $\mathbb{P}(|X|>z)\le Ke^{-Cz^2},\ \forall z>0$, then $\|X\|_{\psi_2}\le\sqrt{(1+K)/C}$.
\end{lemma}

Let $T$ be a semi-metric space with the semi-metric $\sd$ and $\{X_t:t\in T\}$ be a random process indexed by $T$. The random process $\{X_t:t\in T\}$ is called sub-Gaussian if
\beq\label{e41}
\mathbb{P}(|X_s-X_t|>z)\le 2e^{-\frac 12 z^2/\sd(s,t)^2}\ \ \ \ \forall s,t\in T, \ \ z>0.
\eeq
For a semi-metric space $(T,\sd)$ and $\vep>0$, the covering number $N(\vep,T,\sd)$ is the minimum number of $\vep$-balls that cover $T$ and $\log N(\vep,T,\sd)$ is called the covering entropy which is an important quantity to characterize the complexity of the set $T$.  
The following maximal inequality \cite[Section 2.2.1]{Vaart} plays an important role in our analysis.
\begin{lemma}\label{lem:4.2}
If $\{X_t:t\in T\}$ is a separable sub-Gaussian random process, then 
\ben
\|\sup_{s,t\in T}|X_s-X_t|\|_{\psi_2}\le K\int^{\diam\, T}_0\sqrt{\log N\Big(\frac\vep 2,T,\sd\Big)}\ d\vep.
\een
Here $K>0$ is some constant.
\end{lemma}

The following result on the estimation of the covering entropy of Sobolev spaces is due to Birman-Solomyak \cite{Birman}.
\begin{lemma}\label{lem:4.3}
Let $Q$ be the unit square in $\R^d$ and $BW^{\alpha,p}(Q)$ be the unit sphere of the Sobolev space $W^{\alpha,p}(Q)$,
where $\alpha> 0$, $p\ge 1$. Then for $\vep>0$ sufficient small, the entropy
\ben
\log N(\vep, BW^{\alpha,p}(Q), \|\cdot\|_{L^q(Q)})\le C\vep^{-d/\alpha},
\een
where if $\alpha p>d$, $1\le q\le\infty$, otherwise if $\alpha p\le d$, $1\le q < q^*$ with $q^*=p(1-\alpha p/d)^{-1}$.
\end{lemma}

The following lemmas are proved in \cite{Chen-Zhang2022}.
\begin{lemma}\label{lem:4.5}
$\{E_n(u):=(e,Su)_n: u\in H^s(\Om)\}$ is a sub-Gaussian random process with respect to the semi-distance $\sd(u,v)=\|Su-Sv\|_n^*$, where $\|Su\|^*_n:=\sigma n^{-1/2}\|Su\|_n$.
\end{lemma}
\begin{lemma}\label{lem:4.6}
If $X$ is a random variable which satisfies
\ben
\mathbb{P}(|X|>\alpha (1+z))\le C_1e^{-z^2/K_1^2},\ \ \forall\alpha>0, z\ge 1 ,
\een
where $C_1, K_1$ are some positive constants, then $\|X\|_{\psi_2}\le C(C_1,K_1)\alpha$ for some constant $C(C_1,K_1)$ depending only on $C_1,K_1$.
\end{lemma}

\begin{theorem}\label{thm:4.1}
Let $f^{\sigma}\in L^2(\Om)$ be the solution of \ref{p1}. Denote by $\rho_0=\|f^*\|_{H^s(\Omega)}+\sigma n^{-1/2}$. If we take
$\lambda^{1/2+d/4/(2+s)}= O(\sigma n^{-1/2}\rho_0^{-1})$,
then there exists a constant $C>0$ such that
\ben
\mathbb{P}(\|Sf^{\sigma}-Sf^*\|_n\ge \lam^{1/2}\rho_0z)\le Ce^{-Cz^2},\ \ \mathbb{P}(\|f^{\sigma}\|_{L^2(\Om)}\ge \rho_0z)\le Ce^{-Cz^2}.
\een

If $n^{(4+2s)/d}\lambda\geq 1$,  for $s=0$, we have
\begin{equation}\label{P-H-1-1}
\mathbb{P}(\|f^{\sigma}-f^*\|_{(H^{1}(\Omega))^*} \geq C \lambda^{1/4} \rho_0 z ) \leq Ce^{-Cz^2}.
\end{equation}

For $s=1$, we have
\begin{equation}\label{P-L2-1-1}
\mathbb{P}(\|f^{\sigma}-f^*\|_{L^2(\Omega)} \geq C \lambda^{1/6} \rho_0 z ) \leq Ce^{-Cz^2}.
\end{equation}
\end{theorem}
\begin{proof} By \ref{e3} we only need to prove
\beq\label{xx2}
\|\|Sf^{\sigma}-Sf^*\|_n\|_{\psi_2}\le C\lam_n^{1/2}\rho_0,\ \ \||f^{\sigma}|\|_{\psi_2}\le C\rho_0.
\eeq
We will only prove the first estimate in \ref{xx2} by the peeling argument. The other estimate can be proved in a similar way.
It follows from \ref{p7} that
\beq\label{g1}
\|Sf^{\sigma}-Sf^*\|^2_n + \lambda \|f^{\sigma}\|_{L^2(\Omega)}^2 \leq 2(e,Sf^{\sigma}-Sf^*)_n + \lambda \|f^*\|_{L^2(\Omega)}^2.
\eeq
Let $\delta >0,\ \rho>0$ be two constants to be determined later, and
\beq\label{g2}
A_0=[0,\delta), A_i=[2^{i-1}\delta,2 ^i\delta), \ \ B_0=[0,\rho), B_j=[2^{j-1}\rho,2^j\rho),\ \ \ \ i,j\ge 1.
\eeq
For $i,j\ge 0$, define
\ben
F_{ij}= \{v \in L^2(\Omega):~  \|Sv-Sf^*\|_n \in A_i ~,~ \|v\|_{L^2(\Omega)} \in B_j \}.
\een
Then we have
\beq\label{g3}
\mathbb{P}(\|Sf^{\sigma}-Sf^*\|_n>\de)\le\sum_{i=1}^\infty\sum_{j=0}^\infty \mathbb{P}(f^{\sigma}\in F_{ij}).
\eeq
Now we estimate $\mathbb{P}(f^{\sigma}\in F_{ij})$. By Lemma \ref{lem:4.5}, $\{(e,Sv)_n:v\in L^2(\Om)\}$ is a sub-Gaussian random process with respect to the semi-distance $\sd(u,v)=\delta n^{-1/2}\|Su-Sv\|_n$. It is
easy to see that
\ben
\sup_{u,v\in F_{ij}}(\|Sf-Sf^*\|_n+\|Sv-Sf^*\|_n)\le 2\sigma n^{-1/2}\cdot 2^i\de.
\een
Then by the maximal inequality in Lemma \ref{lem:4.2} we have
\ben
\|\sup_{u\in F_{ij}}|(e,Su-Sf^*)_n|\|_{\psi_2}&\le& K\int^{\sigma n^{-1/2}\cdot 2^{i+1}\de}_0\sqrt{\log N\left(\frac\vep 2,F_{ij}, \sd\right)}\,d\vep\nn\\
&=&K\int^{\sigma n^{-1/2}\cdot 2^{i+1}\de}_0\sqrt{\log N\left(\frac\vep{2\sigma n^{-1/2}},F_{ij}, \|\cdot\|_n\right)}\,d\vep.
\een
By Lemma
\ref{lem:4.3} we have the estimate for the entropy
\ben
\log N\left(\frac\vep{2\sigma n^{-1/2}},F_{ij}, \|\cdot\|_n\right)\le \log N\left(\frac\vep{2\sigma n^{-1/2}},F_{ij}, \|\cdot\|_{L^\infty(\Om)}\right)
\le C\left(\frac{2\sigma n^{-1/2}\cdot2^j\rho}{\vep}\right)^{d/(2+s)}.
\een
Therefore,
\be
\|\sup_{u\in F_{ij}}|(e,Su-Sf^*)_n\|_{\psi_2}&\le&K\int^{\sigma n^{-1/2}\cdot 2^{i+1}\de}_0\left(\frac{2\sigma n^{-1/2}\cdot2^j\rho}{\vep}\right)^{d/2/(2+s)}\,d\vep\nn\\
&=&C\sigma n^{-1/2}(2^j\rho)^{d/2/(2+s)}(2^i\de)^{1-d/2/(2+s)}.\label{g5}
\ee
By \ref{g1} and \ref{e3} we have for $i,j \geq 1$:
\ben
\mathbb{P}(f^{\sigma}\in F_{ij})& \leq & \mathbb{P}(2^{2(i-1)}\delta^2 + \lambda 2^{2(j-1)}\rho^2 \leq 2 \mathop {\sup}\limits_{u \in F_{ij}}|(e,u-f^*)_n|
+ \lambda \rho^2_0 ) \\
& =& \mathbb{P}(2 \mathop {\sup}\limits_{u \in F_{ij}}|(e,Su-Sf^*)_n|\ge 2^{2(i-1)}\delta^2 + \lambda 2^{2(j-1)}\rho^2 - \lambda \rho^2_0)\\
&\le &2\exp \left[- \frac{1}{C\sigma^2 n^{-1}} \left(\frac{2^{2(i-1)}\delta^2 + \lambda 2^{2(j-1)}\rho^2-\lambda \rho^2_0}{ (2^i\delta)^{1-d/2/(2+s)} (2^j\rho)^{d/2/(2+s)}} \right)^2\right].
\een
Now we take
\beq\label{g6}
\delta^2=\lambda\rho_0^2 (1+z)^2,\ \rho=\rho_0, \ \ \mbox{where }z \ge 1.
\eeq
Since by assumption $\lambda^{1/2+d/4/(2+s)}=O(\sigma n^{-1/2}\rho_0^{-1})$ and $\sigma n^{-1/2}\rho_0^{-1}\leq 1$, we have
$\lam_n\le C$ for some constant. By some simple calculation we have for $i,j \geq 1$,
\ben
\mathbb{P}(f^{\sigma}\in F_{ij}) \le 2\exp \left[ - C \left(\frac{2^{2(i-1)}z(1+z) + 2^{2(j-1)}}{ (2^i (1+z))^{1-d/2/(2+s)} (2^j)^{d/2/(2+s)}} \right)^2\right].
\een
By using the elementary inequality $ab\le \frac 1p a^p+\frac 1q b^q$ for any $a,b>0, p,q>1, p^{-1}+q^{-1}=1$, we have
$(2^i (1+z))^{1-d/2/(2+s)} (2^j)^{d/2/(2+s)}\le C( (1+z)2^i+2^j)$. Thus
\ben
\mathbb{P}(f^{\sigma}\in F_{ij})\le 2\exp \left[ - C (2^{2i} z^2 + 2^{2j}) \right].
\een
Similarly, one can prove for $i\geq 1, j=0$,
\ben
\mathbb{P}(f^{\sigma}\in F_{i0}) \le 2\exp \left[- C (2^{2i} z^2) \right].
\een
Therefore, since $\sum^\infty_{j=1}e^{-C(2^{2j})}\le e^{-C}< 1$ and $\sum^\infty_{i=1}e^{-C(2^{2i}z^2)}\le e^{-Cz^2}$, we obtain finally
\ben
\sum_{i=1}^\infty\sum_{j=0}^\infty \mathbb{P}(f^{\sigma}\in F_{ij}) \le 2\sum_{i=1}^\infty\sum_{j=1}^\infty e^{- C (2^{2i} z^2 + 2^{2j})}+2\sum^\infty_{i=1}e^{- C (2^{2i} z^2)}\le 4e^{-Cz^2}.
\een
Now inserting the estimate to \ref{g3} we have
\beq\label{g7}
\mathbb{P}(\|Sf^{\sigma}-Sf^*\|_n>\lam^{1/2}\rho_0(1+z))\le 4e^{-Cz^2}\ \ \ \ \forall z\ge 1.
\eeq
This implies by using Lemma \ref{lem:4.6} that $\|\|Sf^{\sigma}-Sf^*\|_n\|_{\psi_2} \leq C \lambda^{1/2}\rho_0$, which is the first estimate in \ref{xx2}.
The proofs of \ref{P-H-1-1} and \ref{P-L2-1-1} are similar to \ref{H-1-1} and \ref{L2-1-1} in Theorem \ref{thm:2.1}. This completes the proof.
\end{proof}

Combining Theorem \ref{thm-stability} and Theorem \ref{thm:4.1}, we have the following distribution estimates with exponential tail for the inverse problem \textbf{P2}.
\begin{theorem}\label{thm-estimate-3.4}
Let $q^*$ be the fixed point of iteration \ref{fix-iteration} and $q^\sigma$ be the fixed point of iteration \ref{fix-iteration-step2}. Then, under the conditions in Theorem \ref{thm-stability}, there exists a constant $C>0$ depending only on $C_p$, $M_b$, $M$, $m_{\mathcal Q}$ and $T$ such that:
\begin{itemize}
\item [(i)] For the case of $s=0$ in \ref{p1}, it holds
$$
\mathbb{P}(\|q^{\sigma}-q^*\|_{(H^1(\Omega))^*} > (\lambda^{1/4}+\lambda^{1/2} \|p\|_{L^\infty(\Omega)} )\rho_0(1+z)) \le C e^{-Cz^2},
$$
where $\lam^{1/2+d/8}=O(\sigma n^{-1/2}\rho_0^{-1})$;

\item [(ii)]  For the case of $s=1$ in \ref{p1}, it holds
$$
\mathbb{P}(\|q^{\sigma}-q^*\|_{L^2(\Omega)} > (\lambda^{1/6}+\lambda^{1/2} \|p\|_{L^\infty(\Omega)} )\rho_0(1+z)) \le C e^{-Cz^2},
$$
where
$\lam^{1/2+d/12}=O(\sigma n^{-1/2}\rho_0^{-1})$.
\end{itemize} 
\end{theorem}

\section{Numerical reconstruction.}\label{section4}
In this section, we present some two-dimensional numerical results to illustrate the theoretical results. We set the discrete measurement points to be scattered but quasi-uniformly distributed in $\Omega$. The discrete noisy data are given by
$$
g_i^\sigma=g(x_i)+e_i, \quad i=1, 2, \cdots, n,
$$
where the noises $\{e_i\}_{i=1}^n$ are normal random variables with variance $\sigma$. We will follow the idea in Section \ref{section3-Stochastic} and design algorithms to solve the two sub-problems {\bf P1} and {\bf P2}. 

\subsection{Algorithms for solving problems {\bf P1} and {\bf P2}.}
Firstly, for {\bf P1}, Theorem \ref{thm:2.1} and Theorem \ref{thm:4.1} suggest that the optimal regularization parameter should be taken as
\beq\label{examo1}
\lam^{1/2+1/(4+2s)}=\sigma n^{-1/2}\|f^*\|_{H^s(\Om)}^{-1},\quad s\in\{0,1\}.
\eeq 
This is a prior estimate with knowledge of the true function $f^*$ and variance $\sigma$. Here we propose a self-consistent algorithm to determine the parameter $\lam$ without knowing $f^*$ and $\sigma$. In the algorithm we estimate $\|f^*\|_{H^s(\Om)}$ by $\|f^\sigma\|_{H^s(\Om}$ and $\sigma$ by $\|Sf^\sigma-g^\sigma\|_n$ since $\|Sf^\sigma-g^\sigma\|_n=\|e\|_n$ provides a good estimation of the variance $\sigma$ by the law of large numbers. Since $f^*$ and the noise level $\sigma$ are unknown, a natural choice for the initial value $\lam_0$ is $\lam_0^{1/2+1/(4+2s)}=n^{-1/2}, s\in\{0,1\}$. We summarize the above strategy as the following Algorithm \ref{alg1}.

\begin{algorithm}[ht]
{\bf Data:} The discrete noisy data $\{g_i^\sigma\}_{i=1}^n$; \\
{\bf Result:} Approximate $f^*$ and $Sf^*$.\\
Set an initial guess $\lam_0^{1/2+1/(4+2s)}=n^{-1/2}, s\in\{0,1\}$; \\
{\bf for} $j=0, 1, \cdots$, {\bf do}\\
Solve \ref{p1} for $f^\sigma$ with $\lam$ replaced by the current value of $\lam_{j}$ on the mesh;\\
Update
$$\lam_{j+1}^{1/2+1/(4+2s)}=\|Sf^\sigma-g^\sigma\|_n n^{-\frac{1}{2}}\|f^\sigma\|_{H^s(\Om)}^{-1};$$\\
{\bf end for} if $|\lam_{j}-\lam_{j+1}|<10^{-10}$;\\
$f^* \leftarrow f^\sigma$, $Sf^* \leftarrow Sf^\sigma$;\\
{\bf output:} The approximated $f^*$ and $Sf^*$.
\caption{A self-consistent algorithm for finding the optimal $\lambda$ in {\bf P1}}\label{alg1}
\end{algorithm} 

Next, for {\bf P2}, we follow the iterative scheme \ref{fix-iteration-step2} and conclude the inversion process as the following Algorithm \ref{alg2}.

\begin{algorithm}[ht]
{\bf Data:} The $f^\sigma$ and $Sf^\sigma$ reconstructed from Algorithm \ref{alg1}; \\
{\bf Result:} Approximate $q^*$.\\
Set an initial guess $q_0^\sigma=\frac{f^\sigma+p(x) Sf^\sigma}{u_e(x,T;0)}$; \\
{\bf for} $j=0, 1, \cdots$, {\bf do}\\
Solve \ref{p1} for $f^\sigma$ with $\lam$ replaced by the current value of $\lam_{j}$ on the mesh;\\
 Update
 $$q_{j+1}^\sigma=\frac{\partial_t u_m(x,T;q_j^\sigma)+f^\sigma+p(x)Sf^\sigma}{u_e(x,T;q_j^\sigma)};$$\\
 {\bf end for} if $\|q_{j+1}^\sigma-q_j^\sigma\|_{L^2(\Omega)}<10^{-10}$;\\
 $q^* \leftarrow q^\sigma\leftarrow q_j^\sigma$;\\
 {\bf output:} The approximated $q^*$.
 \caption{A iterative algorithm for finding the fixed point $q^*$ in {\bf P2}}\label{alg2}
\end{algorithm} 

\subsection{Numerical examples.}
In the following examples, we take the domain $\Omega=(0,1)\times(0,1)$ and $T=1$. We apply the finite element method to solve the direct problems. 
If not specified, we set the mesh size $h=0.01$ and the time step size $\tau=0.01$, which are sufficiently small so that the finite element errors are negligible. 

\begin{example}\label{ex1}
Suppose that the exact solution in \ref{eqn:elliptic} is given by 
\begin{equation}
f^*(x,y)=\sin(2\pi x)\sin(2\pi y), \quad (x,y)\in (0,1)^2. 
\end{equation}
We perform the Algorithm \ref{alg1} for solving the problem {\bf P1}.
\end{example}

For this example, we set the number of observation points $x_i$ to be $n=10^4$. If not specified, we fix the variance $\sigma=0.002$ which implies that the relative noise level $\sigma/\|\mathcal{S}f^*\|_{L^\infty(\Om)}$ is about $10\%$ since $\|Sf^*\|_{L^\infty(\Omega)}\approx 0.02$. To show the numerical accuracy, we define the relative errors below:
\begin{equation}
Err1:= \frac{\|Sf^\sigma-Sf^*\|_n}{\|Sf^*\|_n}, \quad Err2:= \frac{\|f^\sigma-f^*\|_{(H^1(\Omega))^*}}{\|f^*\|_{(H^1(\Omega))^*}}, \quad Err3:= \frac{\|f^\sigma-f^*\|_{L^2(\Omega)}}{\|f^*\|_{L^2(\Omega)}},
\end{equation}
where $f^*$ is the exact solution and $f^\sigma$ is the numerical reconstruction by Algorithm \ref{alg1}. Then, we do the following:

(1) Firstly, we numerically verify the optimal choice of parameter $\lambda$ as in \ref{examo1} and test the efficiency of the Algorithm \ref{alg1} for estimating this optimal $\lambda$. We choose several values of regularization parameter $\lambda$ from $10^{-3}$ to $10^{-10}$, and for each choice, we solve {\bf P1} for its minimizer $f^\sigma$, then plot the logarithmic values of $Err1$, $Err2$ and $Err3$ in Figure \ref{ex1-fig1}, respectively. 

\begin{figure}[h!]
\begin{center}
\includegraphics[width=1\textwidth,height=8cm]{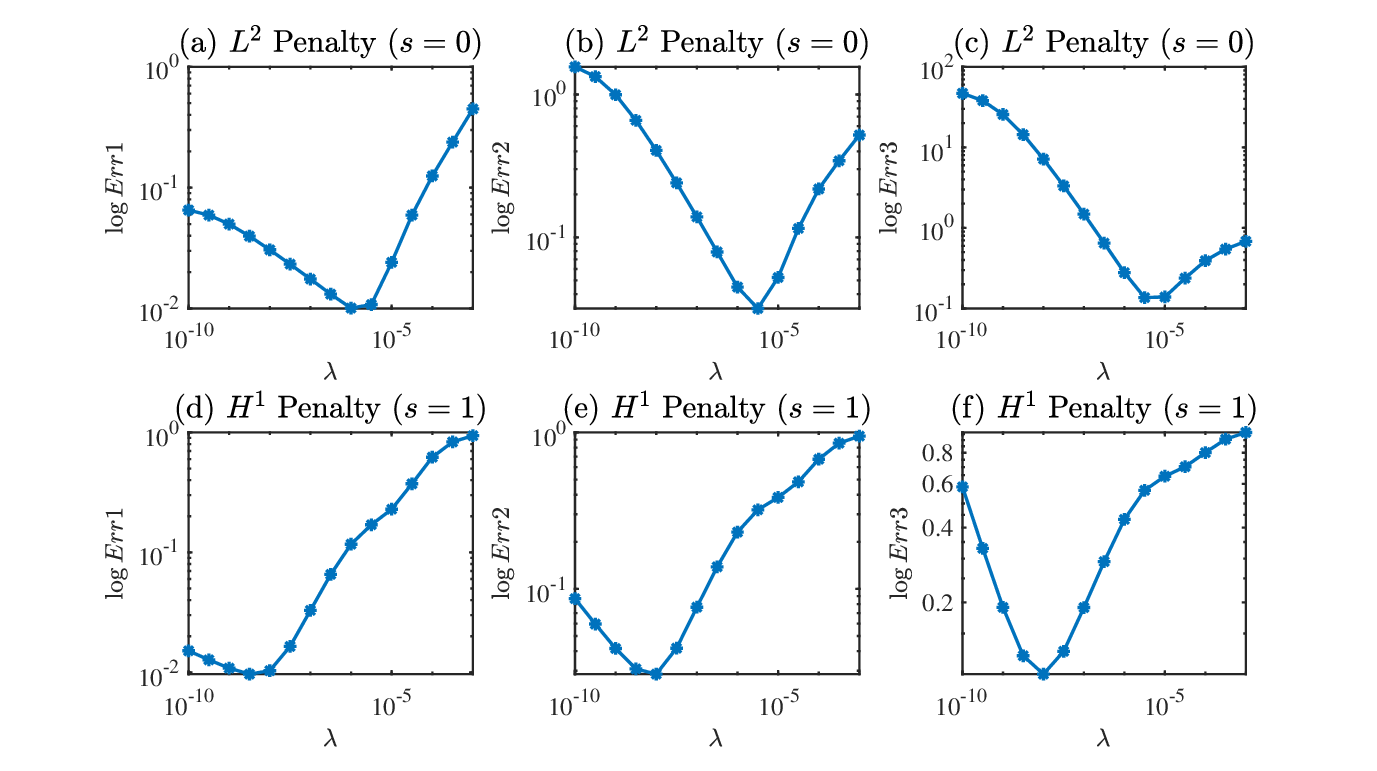}
\caption{Numerical results for different regularization parameter $\lambda$ (Manual test, noise $10\%$).}\label{ex1-fig1}
\end{center}
\end{figure}

(2) Next, we estimate the optimal regularization parameter $\lambda$ automatically using Algorithm \ref{alg1} and show the iteration process in Figure \ref{ex1-fig2}, respectively. For better readability, we summarize the estimations of the nearly optimal regularization parameter $\lambda$ by \ref{examo1}, manual test, and Algorithm \ref{alg1} in Table \ref{ex1-tab1}, respectively.
We show the exact and reconstructed solutions from the data with different noise levels in Figure \ref{ex1-fig3}, respectively.

\begin{figure}[h!]
\begin{center}
\includegraphics[width=1\textwidth,height=10cm]{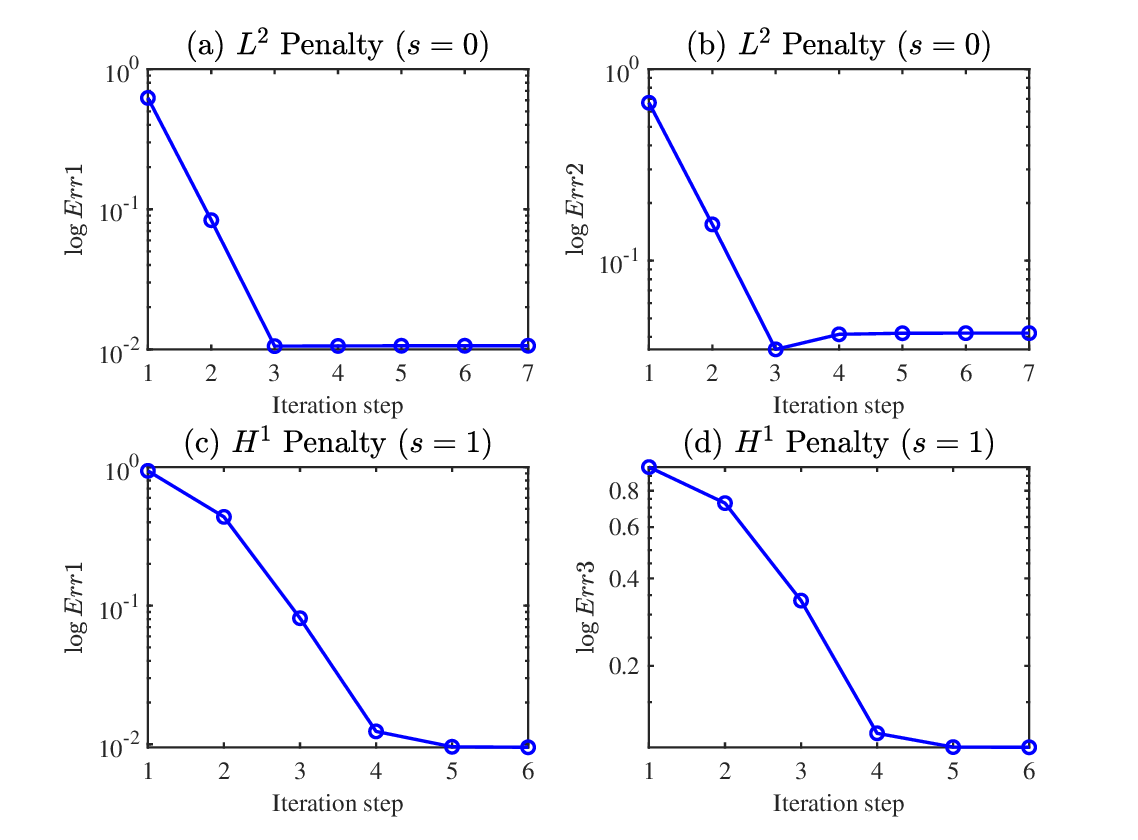}
\caption{The optimal regularization parameter provided by self-consistent Algorithm \ref{alg1}. The empirical errors $\log Err1$, $\log Err2$ and $\log Err3$ by Algorithm \ref{alg1} with respect to iteration number ($s=0$ or $1$).}\label{ex1-fig2}
\end{center}
\end{figure}

\begin{table}[h!]
\caption{Numerical results for  Example \ref{ex1} (noise $10\%$).}\label{ex1-tab1}
\begin{center}
\begin{tabular}{cccccc}
\toprule
   & \text{Method}     &\text{Optimal $\lambda$}  &$Err1$  &$Err2$ &$Err3$\\
\midrule
$L^2$ \text{penalty} $(s=0)$    & \ref{examo1}      &  1.35e-6    &  1.08e-2    & 4.16e-2 & 2.23e-1\\
& Manual test   &  1.00e-6   &  1.01e-2   &  4.49e-2 & 2.79e-1 \\
& Algorithm \ref{alg1}    & 1.31e-6    & 1.06e-2    & 4.18e-2 & 2.40e-1\\
\midrule
$H^1$ \text{penalty} $(s=1)$ & \ref{examo1}   & 9.32e-9     & 1.03e-2  & 2.93e-2  & 1.02e-1 \\
& Manual test       & 3.16e-9     & 9.70e-3 & 3.08e-2 & 1.22e-1  \\
& Algorithm \ref{alg1} & 1.06e-8  & 9.70e-3 & 2.94e-2 & 1.03e-1  \\
\bottomrule
\end{tabular}
\end{center}
\end{table}

(3) Finally, we numerically verify the convergence rates in Theorem \ref{thm:2.1}. We fix $\sigma=0.002$ and choose the number of data $n$ from $10^4$ to $10^6$. We use \ref{examo1} to determine the optimal parameter $\lambda$ for each $n$. Furthermore, for each $n$, we perform the inversion ten times using different random noisy data sets and take the average of ten inversion results. Theorem \ref{thm:2.1} shows that 
\begin{equation}\label{Err}
\begin{cases}
\begin{aligned}
\mathbb{E}(Err1) &\leq C {\lambda}^{1/2}, && s\in\{0,1\},\\
\mathbb{E}(Err2) &\leq C {\lambda}^{1/4}, && s=0,\\
\mathbb{E}(Err3) &\leq C {\lambda}^{1/6}, && s=1,\\
\end{aligned}
\end{cases}
\end{equation}
where $C$ are some constants independent of $n$, $\sigma$ and $\lambda$. We  numerically verify above estimates in Figure \ref{ex1-fig4}, respectively.

\begin{figure}[h!]
\begin{center}
\includegraphics[width=1\textwidth,height=6.5cm]{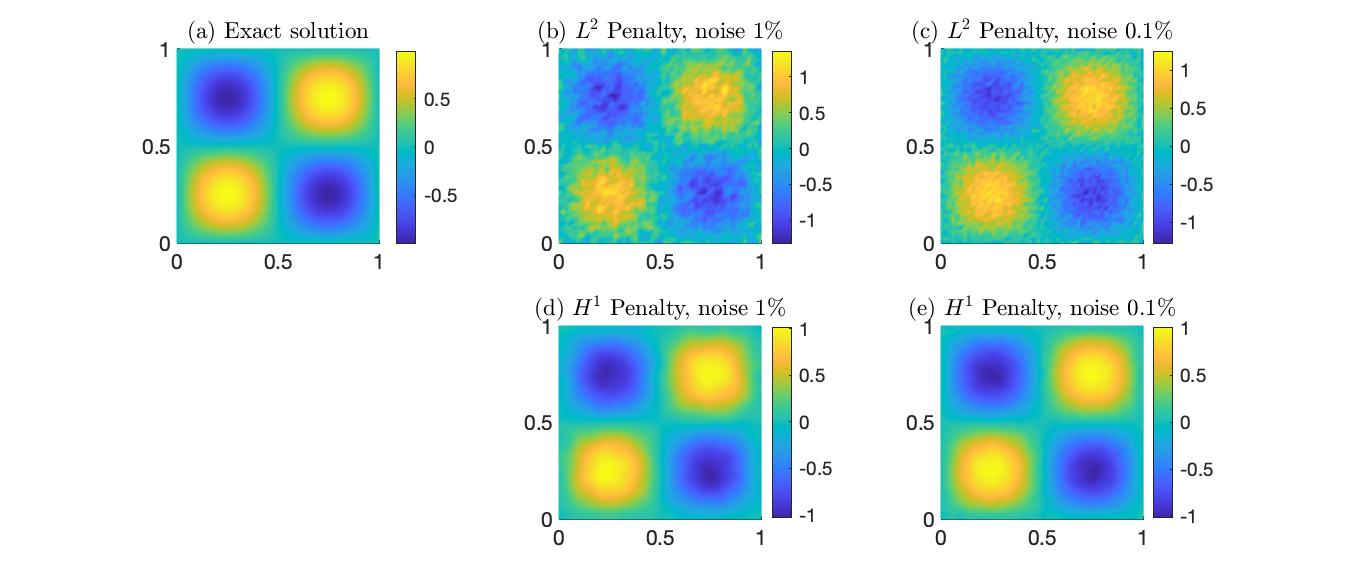}
\caption{The exact and reconstructed solutions by Algorithm \ref{alg1} for Example \ref{ex1}.}\label{ex1-fig3}
\end{center}
\end{figure}

\begin{figure}[h!]
\begin{center}
\includegraphics[width=1\textwidth,height=11cm]{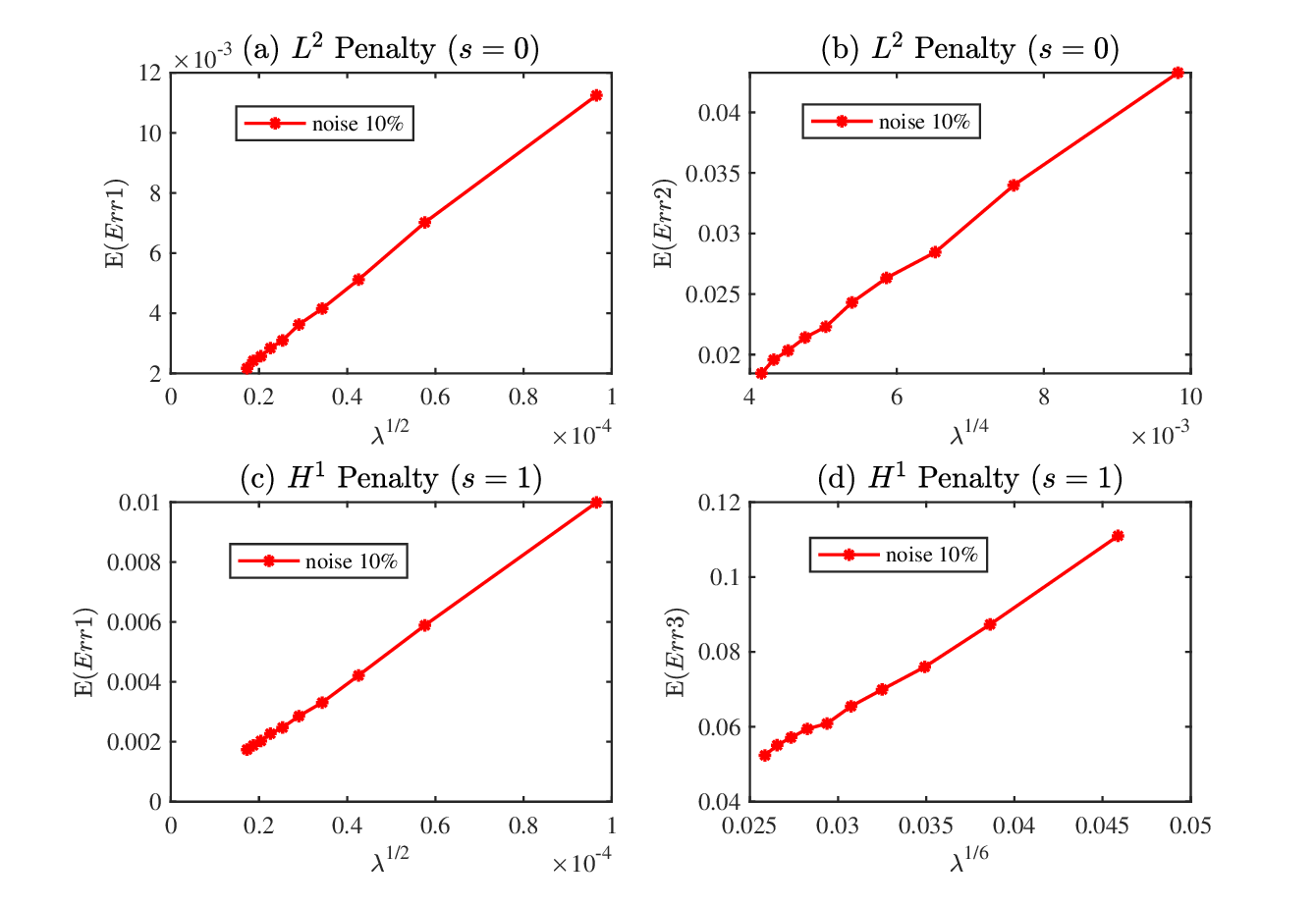}
\caption{(a), (c): The linear dependence of $\mathbb{E}(Err1)$ on $\lambda^{1/2}$ for $s=0$ and $s=1$; (b): The linear dependence of $\mathbb{E}(Err2)$ on $\lambda^{1/3}$ for $s=0$; (d): The linear dependence of $\mathbb{E}(Err3)$ on $\lambda^{1/6}$ for $s=1$.}\label{ex1-fig4}
\end{center}
\end{figure}

It can be seen from Table \ref{ex1-tab1} that the proposed algorithm \ref{alg1} is very effective in determining a nearly optimal regularization parameter $\lambda$ iteratively, without the knowledge of $f^*$ and $\sigma$. In fact,  \ref{examo1} suggests the optimal regularization parameter $\lambda\approx 1.3511e\times 10^{-6}$ and $\lambda\approx 9.3234\times 10^{-9}$ for the minimization problem \ref{p1} with penalty $L^2$ and penalty $H^1$, respectively. The manual numerical test  provides the optimal $\lambda\approx 1.0000\times 10^{-6}$ $(s=0)$ and $\lambda\approx 3.1623\times 10^{-9}$ $(s=1)$.  These approximate values are in fact very close, implying that the estimate \ref{examo1} is valid. Furthermore, Figure \ref{ex1-fig2} clearly shows the convergence of the sequence $\{\lambda_{j}\}$ generated by Algorithm \ref{alg1} and the convergence is very fast. The numerical computation gives $\lambda_{7}= 1.3087\times 10^{-6}$ and $\lambda_{6}=1.0551\times 10^{-8}$ that agree very well with the optimal choices given by \ref{examo1} for $s=0$ and 1, respectively. The proposed Algorithm \ref{alg1} can determine the nearly optimal regularization parameter and, as shown in Figure \ref{ex1-fig3} the reconstructions are satisfactory. Moreover, we can see from Figure \ref{ex1-fig4} that $\mathbb{E}(Err1)$, $\mathbb{E}(Err2)$ and $\mathbb{E}(Err3)$ linearly depend on $\lambda^{1/2}$, $\lambda^{1/4}$ and $\lambda^{1/6}$, respectively. This verifies the conclusions in Theorem \ref{thm:2.1}.

Next, we use the following example to investigate the performance of Algorithm \ref{alg2} for solving the problem $\bf P2$. 

\begin{example}\label{ex2}
Suppose that the problem data in \ref{PDE_ue}-\ref{PDE_um} are given by 
\begin{equation}
b(x,y,t)=(x+y)^2t+5, \quad p(x,y)=x+y+10,\quad (x,y)\in (0,1)^2, \; t\in (0,1). 
\end{equation}
We test the following two exact sources:
\begin{itemize}
\item[(1)] Smooth source:
\begin{equation}\label{ex2-smooth}
q_1^*(x,y)=2+\cos(2\pi x)\cos(2\pi y). 
\end{equation}
\item[(2)] Discontinuous source:
\begin{equation}\label{ex2-nonsmooth}
q_2^*(x,y)=
\begin{cases}
\begin{aligned}
1, & \quad \|(x,y)-(0.3,0.8)\|_2\leq 0.1,\\
1, & \quad \|(x,y)-(0.7,0.8)\|_2\leq 0.1,\\
1, & \quad 0.2\leq\|(x,y)-(0.4,0.5)\|_2\leq 0.3,\\
0, & \quad {\rm else}.
\end{aligned}
\end{cases}
\end{equation}
\end{itemize}
See Figure \ref{ex2-fig5} for the profiles of above two exact solutions.
\end{example}

\begin{figure}[h!]
\begin{center}
\includegraphics[width=1\textwidth,height=4.5 cm]{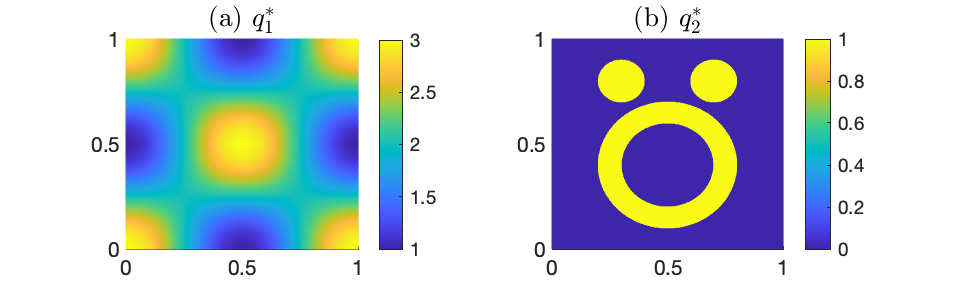}
\caption{Profiles of exact solutions $q_1^*$ and $q_2^*$ in Example \ref{ex2}.}\label{ex2-fig5}
\end{center}
\end{figure}

For this example, we set the number of observation points $x_i$ to be $n=500$ if not specified. To show the numerical accuracy, we define the relative errors below:
\begin{equation}
 Err4:= \frac{\|q^\sigma-q^*\|_{(H^1(\Omega))^*}}{\|q^*\|_{(H^1(\Omega))^*}}, \quad Err5:= \frac{\|q^\sigma-q^*\|_{L^2(\Omega)}}{\|q^*\|_{L^2(\Omega)}}, 
\end{equation}
where $q^*$ is the exact solution and $q^\sigma$ is the numerical reconstruction by Algorithm \ref{alg2}. Then, we do the following:

(1) Firstly, we verify the effectiveness of Algorithm \ref{alg2} for solving {\bf P2}. We fix the initial guess as $q_0^\sigma=(f^\sigma+p(x) Sf^\sigma)/{u_e(x,T;0)}$, where $f^\sigma$ and $Sf^\sigma$ are reconstructed by Algorithm \ref{alg1} from the discrete noise data $\{g_i^\sigma\}_{i=1}^n$. The reconstructions of $q_1^*$ with $L^2$ penalty and  $H^1$ penalty in {\bf P1} are plotted in Figures \ref{ex2-fig6} and \ref{ex2-fig7}, respectively. The reconstructions of $q_2^*$ with $L^2$ penalty and  $H^1$ penalty in {\bf P1} are plotted in Figures \ref{ex2-fig8} and \ref{ex2-fig9}, respectively. We summarize the reconstructions in Table \ref{ex2-tab2}, respectively.

\begin{figure}[h!]
\begin{center}
\includegraphics[width=1\textwidth,height=6.5 cm]{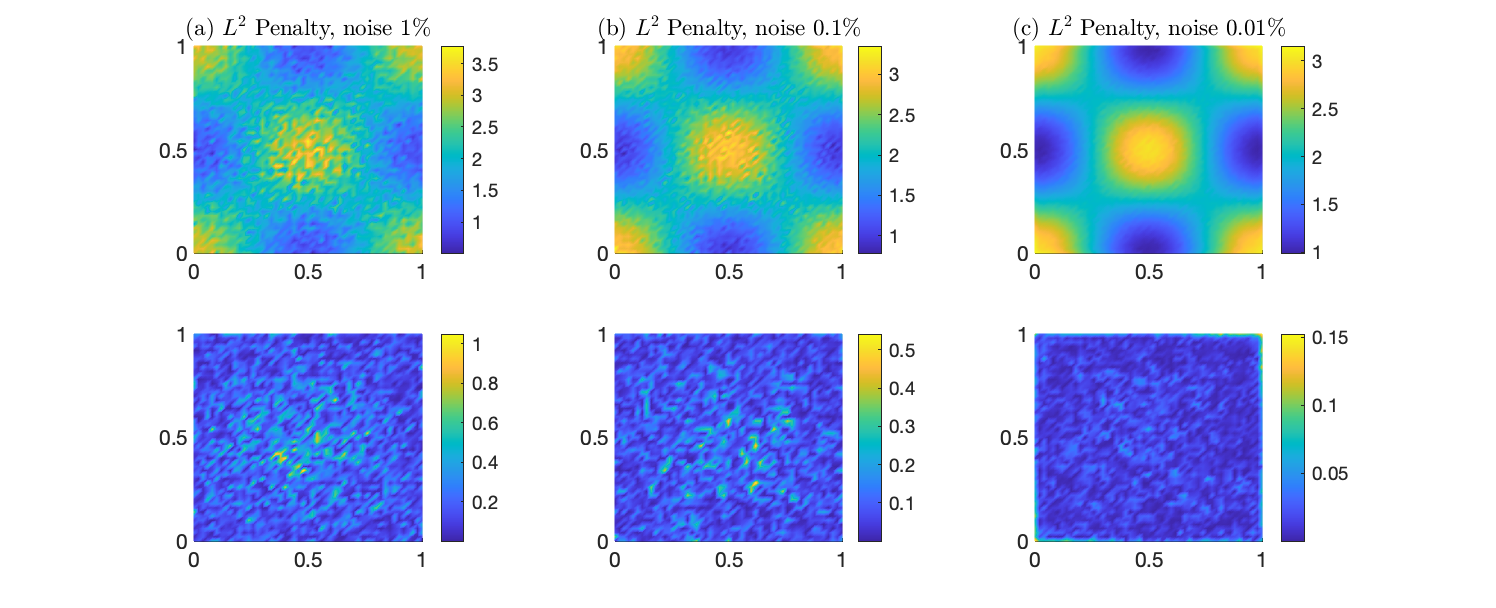}
\caption{Reconstructions of smooth source $q_1^*$ in Example \ref{ex2} with $L^2$ penalty in {\bf P1}. The first raw are profiles of numerical reconstructions $q^\sigma$ and the second raw are their corresponding point error $|q^\sigma-q_1^*|$.}\label{ex2-fig6}
\end{center}
\end{figure}

\begin{figure}[h!]
\begin{center}
\includegraphics[width=1\textwidth,height=6.5 cm]{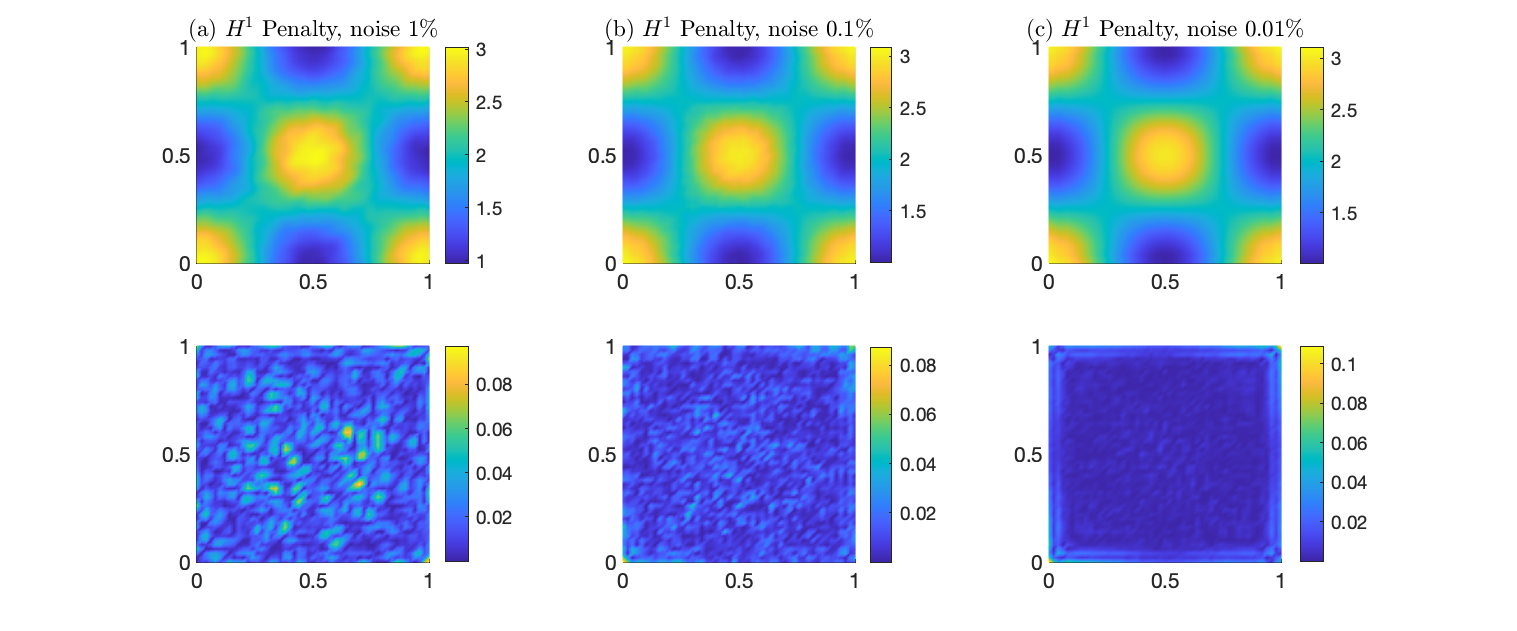}
\caption{Reconstructions of smooth source $q_1^*$ in Example \ref{ex2} with $H^1$ penalty in {\bf P1}. The first raw are profiles of numerical reconstructions $q^\sigma$ and the second raw are their corresponding point error $|q^\sigma-q_1^*|$.}\label{ex2-fig7}
\end{center}
\end{figure}

\begin{figure}[h!]
\begin{center}
\includegraphics[width=1\textwidth,height=6.5 cm]{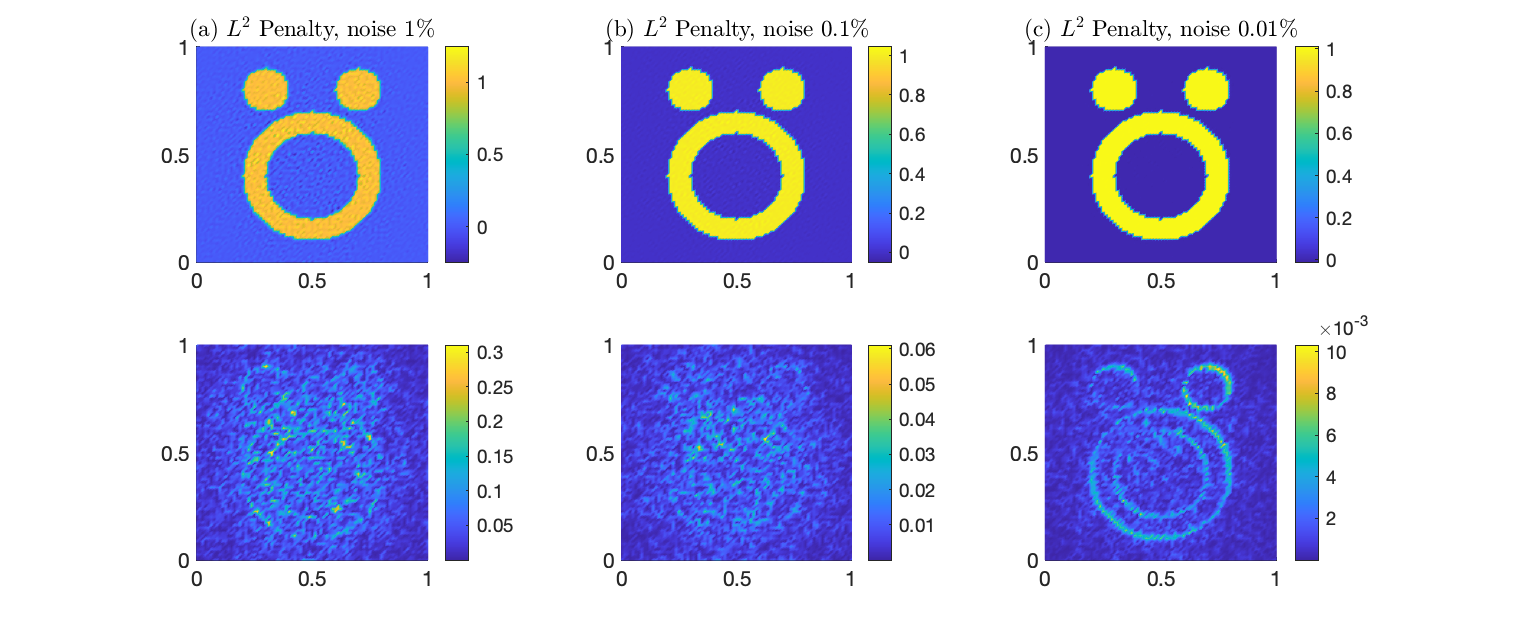}
\caption{Reconstructions of $q_2^*$ in Example \ref{ex2} with $L^2$ penalty in {\bf P1}. The first raw are profiles of numerical reconstructions $q^\sigma$ and the second raw are their corresponding point error $|q^\sigma-q_2^*|$.}\label{ex2-fig8}
\end{center}
\end{figure}

\begin{figure}[h!]
\begin{center}
\includegraphics[width=1\textwidth,height=6.5 cm]{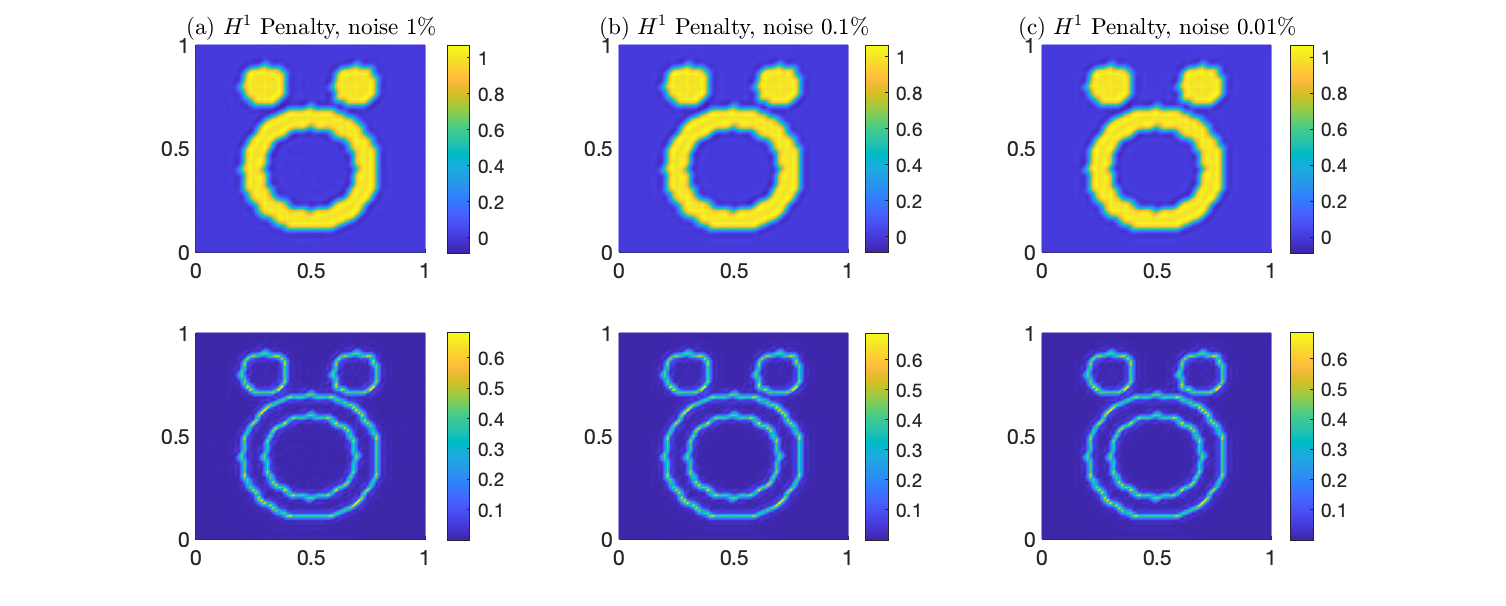}
\caption{The reconstructions of $q_2^*$ in Example \ref{ex2} with $H^1$ penalty in {\bf P1}. The first raw are profiles of numerical reconstructions $q^\sigma$ and the second raw are their corresponding point error $|q^\sigma-q_2^*|$.}\label{ex2-fig9}
\end{center}
\end{figure}

\begin{table}[h!]
\caption{Numerical results for Example \ref{ex2}.}\label{ex2-tab2}
\begin{center}
\begin{tabular}{llllll}
\toprule
Exact source & Noise level & \multicolumn{2}{c}{$L^2$ \text{penalty} in {\bf P1}} & \multicolumn{2}{c}{$H^1$ \text{penalty} in {\bf P1}}\\
\cmidrule(lr){3-4} \cmidrule(lr){5-6}
 & &$Err4$ &$Err5$ &$Err4$ &$Err5$\\
\midrule
\text{Smooth source} $q_1^*$ & 1\% & 1.30e-3 &7.90e-2 & 4.71e-4 &9.80e-3\\
& 0.1\% & 5.01e-4 &3.21e-2 & 3.67e-4 &5.00e-3\\
& 0.01\% & 3.73e-4 &6.70e-3 & 3.63e-4 &3.90e-3\\
\midrule
\text{Discontinuous source} $q_2^*$ & 1\% & 1.00e-3 &7.86e-2 & 1.00e-2 &1.96e-1\\
& 0.1\% &  3.24e-4 & 1.25e-2 & 1.00e-2 &1.95e-1\\
& 0.01\% & 2.94e-4 & 2.80e-3 & 1.00e-2 &1.95e-1\\
\bottomrule
\end{tabular}
\end{center}
\end{table}

(2) Next, we test the convergence of the iteration produced by Algorithm \ref{alg2} with different noise levels and penalties in {\bf P1}. In the experiments, we use the exact solution $q^*=q_1^*$. For the case of using the $L^2$ penalty in {\bf P1}, the logarithmic values of $Err4$ and $Err5$ are plotted in Figure \ref{ex2-fig10}, (a) and (b), respectively. The values for the case of $H^1$ penalty in {\bf P1} are plotted in Figure \ref{ex2-fig10}, (c) and (d), respectively.

\begin{figure}[h!]
\begin{center}
\includegraphics[width=1\textwidth,height=11 cm]{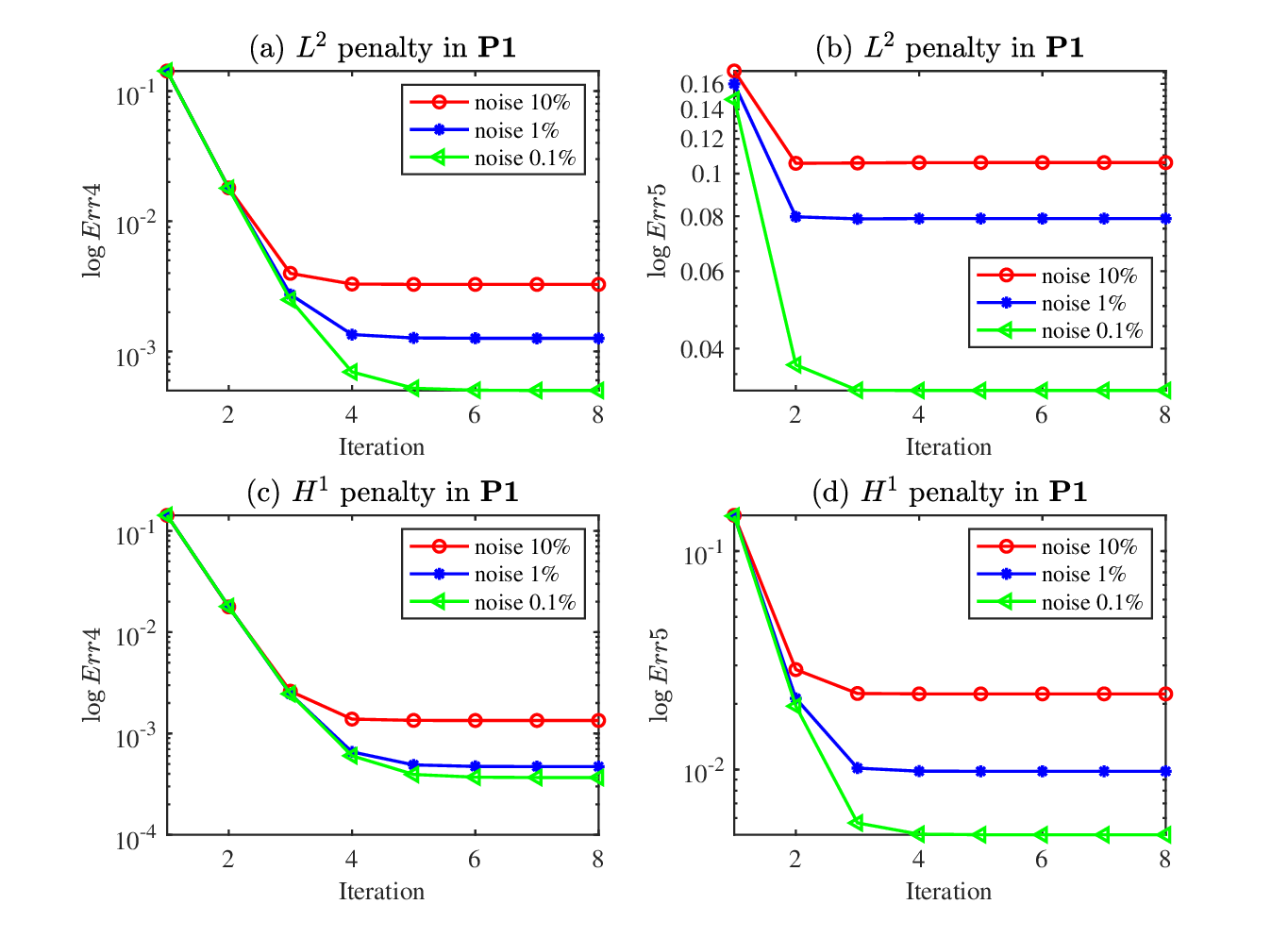}
\caption{Convergence histories of Algorithm \ref{alg2} with different noise levels, where $q^*=q_1^*$ in Example \ref{ex2}.}\label{ex2-fig10}
\end{center}
\end{figure}

(3) Finally, for the case of $q^*=q_1^*$, we numerically verify the convergence rates in Theorem \ref{thm-estimate-3.2}. We fix the noise level as $10\%$ and choose the number of data $n$ from $10^2$ to $10^4$. For each $n$, we use \ref{examo1} to determine the optimal parameter $\lambda$ and perform the inversion ten times using different noisy random data sets. Theorem \ref{thm-estimate-3.2} shows that 
\begin{equation}\label{Err-ex2}
\begin{cases}
\begin{aligned}
\mathbb{E}(Err4) &\leq C {\lambda}^{1/4}, && s=0,\\
\mathbb{E}(Err5) &\leq C {\lambda}^{1/6}, && s=1,\\
\end{aligned}
\end{cases}
\end{equation}
where $C$ are some constants independent of $n$, $\sigma$ and $\lambda$. We verify the above estimates in Figure \ref{ex2-fig11}, respectively. Figure \ref{ex2-fig11}, (a),  (c) show the convergence of ten times inversions for $s=0$ and $1$, respectively, and Figure \ref{ex2-fig11}, (b),  (d) show the convergence of the average of above ten times inversions, respectively.

We clearly observe from the results in Table \ref{ex2-tab2} and Figures \ref{ex2-fig6}-\ref{ex2-fig9} that both the smooth source $q_1^*$ and discontinuous source $q_2^*$ can be reconstructed well using our proposed step-wise inversion algorithms. Besides, the reconstructions show that for smooth source $q_1^*$, the choice of $H^1$ penalty in {\bf P1} could provide better results than the choice of $L^2$ penalty. However, since the nonsmoothness of discontinuous source $q_2^*$, it appears better to choose $L^2$ penalty in {\bf P1} than $H^1$ penalty for reconstructing $q_2^*$. Moreover, as in Figure \ref{ex2-fig10}, our experiments show that for smooth source the Algorithm \ref{alg2} is convergent under either $(H^1(\cdot))^*$ norm or $L^2$ norm with $H^1$ penalty and $L^2$ penalty in {\bf P1}, respectively. Finally, in Figure \ref{ex2-fig11}, under the optimal choice of regularization parameter $\lambda$, we clearly observe that $\mathbb{E}(Err4)$ linearly depends on $\lambda^{1/4}$ and $\mathbb{E}(Err5)$ linearly depends on $\lambda^{1/6}$, respectively. This verifies the convergence rates $\mathbb{E}(\|q^\sigma-q^*\|_{(H^1(\Omega))^*})=O(\lambda^{\frac{1}{4}})$ and $\mathbb{E}(\|q^\sigma-q^*\|_{L^2(\Omega)})=O(\lambda^{\frac{1}{6}})$ in Theorem \ref{thm-estimate-3.2}.

\begin{figure}[ht]
\begin{center}
\includegraphics[width=1\textwidth,height=11 cm]{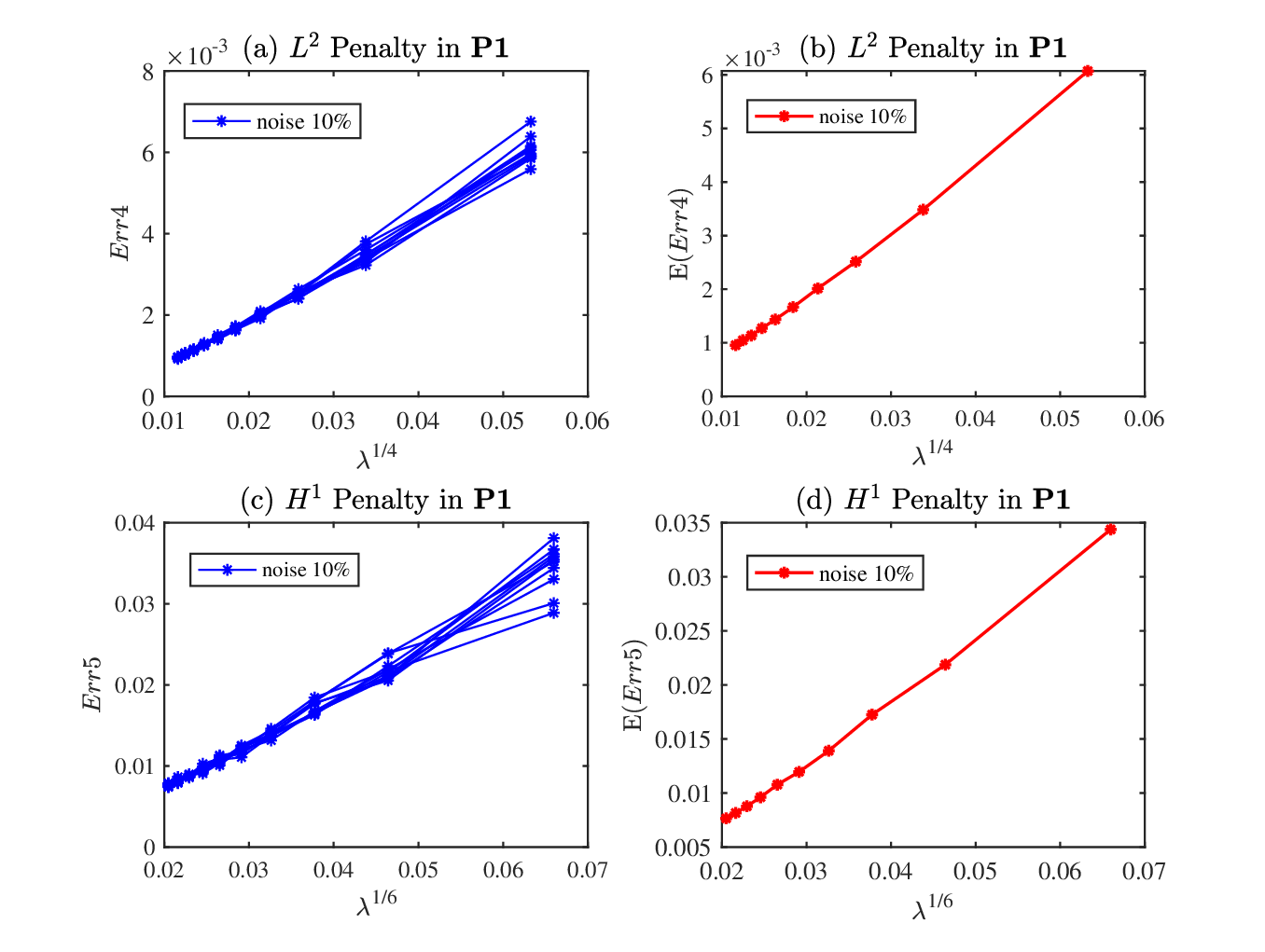}
\caption{(a), (b): The convergence of ten times inversions and the linear dependence of $\mathbb{E}(Err4)$ on $\lambda^{1/4}$ for $L^2$ penalty in {\bf P1} ($s=0$). (c), (d): The convergence of ten times inversions and the linear dependence of $\mathbb{E}(Err5)$ on $\lambda^{1/6}$ for $H^1$ penalty in {\bf P1} ($s=1$).}\label{ex2-fig11}
\end{center}
\end{figure}

\section*{Acknowledgements.}
Zhidong Zhang is supported by the National Key Research and Development Plan of China (Grant No. 2023YFB3002400); Chunlong Sun is supported by the National Natural Science Foundation of China (No.12201298); Wenlong Zhang is supported by the National Natural Science Foundation of China under grant numbers No.12371423 and No.12241104.

\bibliographystyle{plain} 
\bibliography{ref}

\begin{thebibliography}{10}

\bibitem{new1}
Batoul Abdelaziz, Abdellatif El~Badia, and Ahmad El~Hajj.
\newblock Reconstruction of extended sources with small supports in the
  elliptic equation {$\Delta u+\mu u=F$} from a single {C}auchy data.
\newblock {\em C. R. Math. Acad. Sci. Paris}, 351(21-22):797--801, 2013.

\bibitem{Arridge99}
Simon~R. Arridge.
\newblock Optical tomography in medical imaging.
\newblock {\em Inverse Problems}, 15(2):R41--R93, 1999.

\bibitem{Arr09}
Simon~R. Arridge and John~C. Schotland.
\newblock Optical tomography: forward and inverse problems.
\newblock {\em Inverse Problems}, 25(12):123010, 59, 2009.

\bibitem{AB01}
Juliana Atmadja and Amvrossios~C. Bagtzoglou.
\newblock State of the art report on mathematical methods for groundwater
  pollution source identification.
\newblock {\em Environmental Forensics}, 2(3):205--214, 2001.

\bibitem{new3}
A.~El Badia, T.~Ha Duong, and F.~Moutazaim and.
\newblock Numerical solution for the identification of source terms from
  boundary measurements.
\newblock {\em Inverse Problems in Engineering}, 8(4):345--364, 2000.

\bibitem{new2}
A.~El Badia, A.~El Hajj, M.~Jazar, and H.~Moustafa.
\newblock Lipschitz stability estimates for an inverse source problem in an
  elliptic equation from interior measurements.
\newblock {\em Applicable Analysis}, 95(9):1873--1890, 2016.

\bibitem{BaoLiZhao:2020}
Gang Bao, Peijun Li, and Yue Zhao.
\newblock Stability for the inverse source problems in elastic and
  electromagnetic waves.
\newblock {\em J. Math. Pures Appl. (9)}, 134:122--178, 2020.

\bibitem{ZCP01}
Zhu Bing-Quan, Chen Yu-Wei, and Peng Jian-Hua.
\newblock Lead isotope geochemistry of the urban environment in the pearl river
  delta.
\newblock {\em Applied Geochemistry}, 16(4):409--417, 2001.

\bibitem{Birman}
M.~\v~S. Birman and M.~Z. Solomjak.
\newblock Piecewise polynomial approximations of functions of classes
  {$W\sb{p}{}\sp{\alpha }$}.
\newblock {\em Mat. Sb. (N.S.)}, 73(115):331--355, 1967.

\bibitem{Chen-Zhang}
Zhiming Chen, Rui Tuo, and Wenlong Zhang.
\newblock Stochastic convergence of a nonconforming finite element method for
  the thin plate spline smoother for observational data.
\newblock {\em SIAM J. Numer. Anal.}, 56(2):635--659, 2018.

\bibitem{Chen-Zhang2022}
Zhiming Chen, Wenlong Zhang, and Jun Zou.
\newblock Stochastic convergence of regularized solutions and their finite
  element approximations to inverse source problems.
\newblock {\em SIAM J. Numer. Anal.}, 60(2):751--780, 2022.

\bibitem{ChengYamamoto:2022}
Jin Cheng and Masahiro Yamamoto.
\newblock Continuation of solutions to elliptic and parabolic equations on
  hyperplanes and application to inverse source problems.
\newblock {\em Inverse Problems}, 38(8):Paper No. 085005, 23, 2022.

\bibitem{DingGongLiuLo:2024}
Ming-Hui Ding, Rongfang Gong, Hongyu Liu, and Catharine W.~K. Lo.
\newblock Determining sources in the bioluminescence tomography problem.
\newblock {\em Inverse Problems}, 40(12):Paper No. 125022, 28, 2024.

\bibitem{Badia2011}
Abdellatif El~Badia and Takaaki Nara.
\newblock An inverse source problem for {H}elmholtz's equation from the
  {C}auchy data with a single wave number.
\newblock {\em Inverse Problems}, 27(10):105001, 15, 2011.

\bibitem{Evans:1998}
Lawrence~C. Evans.
\newblock {\em Partial differential equations}, volume~19 of {\em Graduate
  Studies in Mathematics}.
\newblock American Mathematical Society, Providence, RI, 1998.

\bibitem{Fleckinger}
Jacqueline Fleckinger and Michel~L. Lapidus.
\newblock Eigenvalues of elliptic boundary value problems with an indefinite
  weight function.
\newblock {\em Trans. Amer. Math. Soc.}, 295(1):305--324, 1986.

\bibitem{FuZhang:2021}
Shubin Fu and Zhidong Zhang.
\newblock Application of the generalized multiscale finite element method in an
  inverse random source problem.
\newblock {\em J. Comput. Phys.}, 429:Paper No. 110032, 17, 2021.

\bibitem{new4}
Galina~C. Garcia, Axel Osses, and Marcelo Tapia.
\newblock A heat source reconstruction formula from single internal
  measurements using a family of null controls.
\newblock {\em J. Inverse Ill-Posed Probl.}, 21(6):755--779, 2013.

\bibitem{Geer}
S.A. Geer.
\newblock {\em Empirical Processes in M-Estimation}.
\newblock Cambridge Series in Statistical and Probabilistic Mathematics.
  Cambridge University Press, 2000.

\bibitem{GER83}
Steven~M. Gorelick, Barbara Evans, and Irwin Remson.
\newblock Identifying sources of groundwater pollution: An optimization
  approach.
\newblock {\em Water Resources Research}, 19(3):779--790, 1983.

\bibitem{Isakov:1990}
Victor Isakov.
\newblock {\em Inverse source problems}, volume~34 of {\em Mathematical Surveys
  and Monographs}.
\newblock American Mathematical Society, Providence, RI, 1990.

\bibitem{new6}
Victor Isakov.
\newblock {\em Inverse problems for partial differential equations}, volume 127
  of {\em Applied Mathematical Sciences}.
\newblock Springer-Verlag, New York, 1998.

\bibitem{Isakov2013}
Victor Isakov, Shingyu Leung, and Jianliang Qian.
\newblock A three-dimensional inverse gravimetry problem for ice with snow
  caps.
\newblock {\em Inverse Probl. Imaging}, 7(2):523--544, 2013.

\bibitem{IsakovLu:2020}
Victor Isakov and Shuai Lu.
\newblock On the inverse source problem with boundary data at many wave
  numbers.
\newblock In {\em Inverse problems and related topics}, volume 310 of {\em
  Springer Proc. Math. Stat.}, pages 59--80. Springer, Singapore, [2020]
  \copyright 2020.

\bibitem{JiangLiYamamoto:2024}
Daijun Jiang, Zhiyuan Li, and Masahiro Yamamoto.
\newblock Coercivity-based analysis and its application to an inverse source
  problem for a subdiffusion equation with time-dependent principal parts.
\newblock {\em Inverse Problems}, 40(12):Paper No. 125027, 15, 2024.

\bibitem{LassasLiZhang:2023}
Matti Lassas, Zhiyuan Li, and Zhidong Zhang.
\newblock Well-posedness of the stochastic time-fractional diffusion and wave
  equations and inverse random source problems.
\newblock {\em Inverse Problems}, 39(8):Paper No. 084001, 36, 2023.

\bibitem{LiLiWang:2022}
Jianliang Li, Peijun Li, and Xu~Wang.
\newblock Inverse source problems for the stochastic wave equations: far-field
  patterns.
\newblock {\em SIAM J. Appl. Math.}, 82(4):1113--1134, 2022.

\bibitem{LinOuZhangZhang:2024}
Guang Lin, Na~Ou, Zecheng Zhang, and Zhidong Zhang.
\newblock Restoring the discontinuous heat equation source using sparse
  boundary data and dynamic sensors.
\newblock {\em Inverse Problems}, 40(4):Paper No. 045014, 17, 2024.

\bibitem{LinZhangZhang:2022}
Guang Lin, Zecheng Zhang, and Zhidong Zhang.
\newblock Theoretical and numerical studies of inverse source problem for the
  linear parabolic equation with sparse boundary measurements.
\newblock {\em Inverse Problems}, 38(12):Paper No. 125007, 28, 2022.

\bibitem{liu16}
Chein-Shan Liu.
\newblock An integral equation method to recover non-additive and non-separable
  heat source without initial temperature.
\newblock {\em International Journal of Heat and Mass Transfer}, 97:943--953,
  2016.

\bibitem{Liu:2020}
Jijun Liu, Manabu Machida, Gen Nakamura, Goro Nishimura, and Chunlong Sun.
\newblock On fluorescence imaging: The diffusion equation model and recovery of
  the absorption coefficient of fluorophores.
\newblock {\em Science China Mathematics}, 65(6):1179--1198, 2022.

\bibitem{nelson20}
P.A. NELSON and S.H. YOON.
\newblock Estimation of acoustic source strength by inverse methods: Part i,
  conditioning of the inverse problem.
\newblock {\em Journal of Sound and Vibration}, 233(4):639--664, 2000.

\bibitem{RundellZhang:2020}
William Rundell and Zhidong Zhang.
\newblock On the identification of source term in the heat equation from sparse
  data.
\newblock {\em SIAM Journal on Mathematical Analysis}, 52(2):1526--1548, 2020.

\bibitem{SunZhang:2022}
Chunlong Sun and Zhidong Zhang.
\newblock Uniqueness and numerical inversion in the time-domain fluorescence
  diffuse optical tomography.
\newblock {\em Inverse Problems}, 38(10):Paper No. 104001, 23, 2022.

\bibitem{tadi97}
M.~Tadi.
\newblock Inverse heat conduction based on boundary measurement.
\newblock {\em Inverse Problems}, 13(6):1585--1605, 1997.

\bibitem{Vaart}
A.~van~der Vaart and J.A. Wellner.
\newblock {\em Weak Convergence and Empirical Processes: With Applications to
  Statistics}.
\newblock Springer Series in Statistics. Springer, 1996.

\bibitem{WangXuZhao:2024}
Tianjiao Wang, Xiang Xu, and Yue Zhao.
\newblock Stability for a multi-frequency inverse random source problem.
\newblock {\em Inverse Problems}, 40(12):Paper No. 125029, 25, 2024.

\bibitem{ZhangWuGuo:2024}
Deyue Zhang, Yue Wu, and Yukun Guo.
\newblock Imaging an acoustic obstacle and its excitation sources from
  phaseless near-field data.
\newblock {\em Inverse Probl. Imaging}, 18(4):797--812, 2024.

\end{thebibliography}
\end{document}